\documentclass[12pt,twoside]{amsart}
\usepackage{amssymb,amsmath,amsthm, amscd, enumerate, mathrsfs}
\usepackage{graphicx, hhline}
\usepackage[all]{xy}
\usepackage{braket}
\usepackage[usenames]{color}
\usepackage{hyperref}
\usepackage{fancyhdr}
\usepackage[top=30truemm,bottom=30truemm,left=30truemm,right=30truemm]{geometry}

\hypersetup{colorlinks=true}

\title{Full termination of MMP for klt pairs with big boundaries}
\author{Kenta Hashizume}
\date{2026/12/01}
\keywords{termination of minimal model program, log minimal model, Mori fiber space}
\subjclass[2020]{14E30}
\address{Department of 
Mathematics, Faculty of Science, Niigata University, Niigata 950-2181, Japan}
\address{Institute for Research Administration, Niigata University, Niigata 950-2181, Japan}
\email{hkenta@math.sc.niigata-u.ac.jp}

\newtheorem{thm}{Theorem}[section]

\newtheorem{lem}[thm]{Lemma}
\newtheorem{cor}[thm]{Corollary}
\newtheorem{prop}[thm]{Proposition}

\theoremstyle{definition}
\newtheorem{defn}[thm]{Definition}
\newtheorem{rem}[thm]{Remark}

\newtheorem{exam}[thm]{Example}
\newtheorem{idea}[thm]{Rough idea}

\newtheorem*{b-divisor}{b-divisors} 
 
\newtheorem*{g-pair}{Generalized pairs} 
\newtheorem*{adj-g-pair}{Divisorial adjunction for generalized pairs} 
\newtheorem*{mmp-g-pair}{MMP for generalized pairs}

\newtheorem{setup}[thm]{Setup}

\newtheorem{step1}{Step}

\newtheorem*{claim*}{Claim}
\begin{document}

\begin{abstract}
We prove the full termination of the minimal model program for klt pairs with big boundaries. 
More strongly, we establish a uniform bound on the number of steps of all MMP starting from a fixed pair. 
In the course of the proof, we introduce extremal-ray functions and show that only finitely many of them can occur along these MMP. 
We also prove the finiteness of marked Mori fiber spaces. 
\end{abstract}

\maketitle

\tableofcontents

\section{Introduction}

We work over the field of complex numbers. 

In this paper, we prove the following results.

\begin{thm}[Full termination]\label{thm--full-termination}
Let $(X,\Delta)$ be a projective klt pair such that $K_{X}+\Delta$ or $\Delta$ is big. 
Then every $(K_{X}+\Delta)$-MMP terminates. 
More strongly, there exists $N \in \mathbb{Z}_{>0}$, depending only on $(X,\Delta)$, such that every $(K_X+\Delta)$-MMP terminates after at most $N$ steps.
\end{thm}

\begin{thm}[Finiteness of marked Mori fiber spaces]\label{thm--fin-mfs}
Let $(X,\Delta)$ be a projective $\mathbb{Q}$-factorial klt pair such that $\Delta$ is big and $K_{X}+\Delta$ is not pseudo-effective. 
Then the set of marked Mori fiber spaces arising from sequences of steps of a $(K_X+\Delta)$-MMP is finite up to isomorphisms compatible with the markings. 
\end{thm}

\subsection{Comparison with previous results}
The termination of the minimal model program (MMP, for short) is one of the most fundamental open problems in birational geometry. 

In dimension three, the MMP was established through the works of Mori, Shokurov, Kawamata, and others (cf.~\cite{mori-threefolds}, \cite{mori-flip}, \cite{shokurov-nonvanishing}, \cite{shokurov-flip}, \cite{shokurov-log-models}, \cite{kawamata-termination}). 
In \cite{mori-threefolds}, Mori developed the theory of extremal rays and their contractions for smooth projective threefolds. 
Shokurov \cite{shokurov-nonvanishing} introduced the notion of difficulty and proved that there are no infinite sequences of three-dimensional canonical flips, although the existence of flips was not known at the time. 
After that, Mori \cite{mori-flip} finally proved the existence of flips for $\mathbb{Q}$-factorial terminal threefolds and thereby completed the threefold MMP. 
This was a major milestone in birational geometry. 
Further, the MMP was developed for pairs consisting of a variety and a boundary divisor. 
In \cite{kawamata-termination}, Kawamata proved the termination of three-dimensional log flips for $\mathbb{Q}$-factorial weak log terminal pairs with $\mathbb{Q}$-boundaries, whereas Shokurov \cite{shokurov-flip} proved the existence of three-dimensional log flips. 
Later, Shokurov \cite{shokurov-log-models} extended the three-dimensional log MMP to lc pairs with $\mathbb{R}$-boundaries. 

In dimension four, Kawamata--Matsuda--Matsuki \cite{kawamata-matsuda-matsuki} proved the termination of terminal flips by making use of Shokurov's difficulty.
Shokurov \cite{shokurov1} later proved the existence of four-dimensional log flips.
Fujino (\cite{fujino-fourfold}, \cite{fujino-fourfold-addendum}) proved the termination of fourfold canonical flips.
In \cite{alexeev-hacon-kawamata}, further termination results for klt fourfolds were proved by Alexeev--Hacon--Kawamata when the anti-log canonical divisor is numerically equivalent to an effective divisor. 
Separately, Birkar \cite{birkar-termination-fourfold} proved termination for klt pairs of Kodaira dimension at least two. 
More recently, Moraga \cite{moraga-fourfold} proved the termination of pseudo-effective lc fourfold flips. 

In higher dimensions, Birkar \cite{birkar-termination} developed an inductive approach towards proving the termination of log flips. 
Birkar--Cascini--Hacon--M\textsuperscript{c}Kernan \cite{bchm} proved the termination of MMP with scaling for klt pairs of general type and for klt pairs with big boundaries. 
More recently, Han--Liu \cite{han-liu-GT} studied relations between the termination of flips for klt pairs and that for terminal pairs. 
In  \cite{han-qi-zhuang}, Han--Qi--Zhuang proved several boundedness results for arbitrary MMP in the same settings as in \cite{bchm}, including a uniform bound on the Cartier indices of Weil divisors.
Building on results in \cite{han-qi-zhuang} and Birkar's inductive approach \cite{birkar-termination}, Han--Liu--Zhuang \cite{han-liu-zhuang} proved effective termination in dimension at most five for klt pairs of general type and for pseudo-effective klt pairs with big boundaries, and obtained explicit bounds for several classes of MMP in lower dimensions. 
Before the present work, however, full termination in the general type and big-boundary settings was not known in arbitrary dimension. 

Theorem \ref{thm--full-termination} proves full termination, and more strongly, effective termination in the sense of \cite[Theorem 1.1]{han-liu-zhuang}, in arbitrary dimension for klt pairs of general type and for klt pairs with big boundaries.
In particular, every MMP starting from a smooth projective variety of general type terminates (see Corollary \ref{cor--intro-termi-sm-gentype} below). 
We emphasize that Theorem \ref{thm--full-termination} also treats the non-pseudo-effective case with a big boundary.

Theorem \ref{thm--fin-mfs} concerns a finiteness problem different from the finiteness questions arising in the Sarkisov program (cf.~\cite{sarkisov}, \cite{corti-sarkisov}, \cite{bruno-matsuki}, \cite{hacon-mckernan-sarkisov}, \cite{he-strong-sarkisov})
and the geography of log models (cf.~\cite{shokurov-log-models}, \cite{bchm}, \cite{shokurov-choi}). 
In Theorem \ref{thm--fin-mfs}, the starting pair is fixed, and we prove the finiteness of marked Mori fiber spaces obtained by running the MMP from this fixed pair. 

Our proof does not involve minimal log discrepancies (mld, for short).
Shokurov showed that the ascending chain condition and the lower semicontinuity conjectures for mld imply the termination of flips
(\cite{shokurovV}, see also \cite[Theorem 3.9]{flipsandflops}). 
On the other hand, our proof essentially uses the existence of log minimal models and Mori fiber spaces for the pairs under consideration.
Therefore, the proof in this paper does not apply to full termination for arbitrary klt pairs. 
Nevertheless, a key ingredient in our proof is Theorem \ref{thm--intro-ind-finiteness}, which gives a general finiteness result for marked small birational models satisfying a uniform klt condition.
Using this theorem, we prove in Theorem \ref{thm--finite-totalMMP} the finiteness of marked models arising after finitely many steps of MMP starting from a fixed pair. 
This finiteness statement is stronger than Theorem \ref{thm--full-termination} and Theorem \ref{thm--fin-mfs}.
However, such a finiteness statement fails in general outside our setting, since there can already be infinitely many possible first MMP steps (see, for example, \cite[Example~4-2-3]{kawamata-matsuda-matsuki}).

\subsection{Strategy of the proof}
We briefly explain how the main ingredients are used. 

The first part of the proof is based on the uniform bound on the Cartier indices along arbitrary MMP, which was proved by Han--Qi--Zhuang \cite[Theorem 1.1]{han-qi-zhuang}. 
Their proof uses the log canonical volume (\cite[Definition 3.9]{han-qi-zhuang}), whose monotonicity along the MMP provides a uniform positive lower bound for local volumes, and the uniform bound on Cartier indices follows from an estimate of Xu--Zhuang (\cite[Corollary~1.4]{xu-zhuang}). 

Combined with the theory of lengths of extremal rays, this bound allows us to carry out an argument using Shokurov polytopes simultaneously on all pairs appearing in sequences of steps of MMP starting from a fixed pair. 
More precisely, we first prove that the negativity of every extremal ray occurring in such an MMP is preserved under a sufficiently small perturbation of the boundary. 
After that, using a rational polytope of boundaries and the uniform bound on the Cartier indices, we show that only finitely many extremal-ray functions (Definition \ref{defn--intnum-func}) can occur. 
Although the extremal rays occur on different birational models, the corresponding functions are all defined on the real N\'{e}ron--Severi space of the original variety. 
The finiteness of these functions, in turn, implies that only finitely many prime divisors on the original variety can be contracted by sequences of steps of the MMP. These arguments are carried out in Section \ref{sec--fin-ext-ray-func}. 

The second part of the proof is an induction on the dimension. In Section \ref{sec--base}, we fix one Mori fiber space and show that the bases of a certain set of Mori fiber spaces form a family of small birational models satisfying a uniform klt condition. This motivates the following key inductive statement.

\begin{thm}[{cf.~Theorem \ref{thm--ind-finiteness}}]\label{thm--intro-ind-finiteness}
Let $Y_{0}$ be a projective $\mathbb{Q}$-factorial variety, and let $\mathfrak{F}_{Y_{0}}$ be a set consisting of small birational maps
$$\psi\colon Y_{0}\dashrightarrow Y$$
to projective $\mathbb{Q}$-factorial varieties $Y$. 
Assume that ${\rm id}_{Y_{0}}\in\mathfrak{F}_{Y_{0}}$ and that, for any effective $\mathbb{R}$-divisor $L_{Y_{0}}$ on $Y_{0}$, there exists $\alpha\in\mathbb{Q}_{>0}$ such that
$$\bigl(Y,2\alpha\psi_{*}G_{Y_{0}}\bigr)$$
is klt for every $G_{Y_{0}}\in |L_{Y_{0}}|_{\mathbb{R}}$ and every $\psi\colon Y_{0}\dashrightarrow Y$ in $\mathfrak{F}_{Y_{0}}$. Then $\mathfrak{F}_{Y_{0}}$ is finite up to isomorphisms compatible with the markings.
\end{thm}

To prove this theorem, we choose an effective big $\mathbb{Q}$-divisor $L_{Y_{0}}$ on $Y_{0}$, and for each $\psi\colon Y_{0}\dashrightarrow Y$ in $\mathfrak{F}_{Y_{0}}$ we put $\Delta_{Y}:=\alpha\psi_{*}L_{Y_{0}}$. Then the log canonical volume gives a uniform bound on the Cartier indices of Weil divisors along every MMP starting from $(Y,\Delta_{Y})$. This bound is independent of the choice of $Y$. Therefore, the argument in Section \ref{sec--fin-ext-ray-func} can be applied uniformly to this set, and we obtain finitely many extremal-ray functions and finitely many prime divisors which may be contracted.
Using these results, we carry out an induction on the dimension of $Y_{0}$. Marked log minimal models form a finite set by \cite[Theorem E]{bchm}. 
For Mori fiber spaces, after fixing the contracted divisors and the extremal-ray function, their bases form a set of lower-dimensional small birational models satisfying the same uniform klt condition. Therefore, the induction hypothesis gives finiteness of the possible resulting models. Section \ref{sec--inv-mmp} gives a finiteness result for inverse MMP steps. 
In the same section, Proposition \ref{prop--num-step-mmp} allows us to reach a resulting model by a uniformly bounded number of steps. 
Hence the finiteness of the resulting models implies the finiteness of the possible starting models. This proves Theorem \ref{thm--intro-ind-finiteness}, from which the main theorems follow.

\subsection{Corollaries and applications}

We discuss several corollaries and applications of the main results of this paper. 
All the applications below rely on the full termination of MMP, rather than merely the termination of MMP with scaling proved in \cite{bchm}.

\subsubsection{Immediate corollaries}
We first collect immediate corollaries of Theorem \ref{thm--full-termination}. 

\begin{cor}\label{cor--intro-termi-sm-gentype}
Let $X$ be a smooth projective variety such that $K_{X}$ is big. 
Then there exists $N \in \mathbb{Z}_{>0}$, depending only on $X$, such that every $K_{X}$-MMP terminates after at most $N$ steps.
\end{cor}

\begin{cor}\label{cor--intro-termi-big-antilog}
Let $X$ be a normal projective $\mathbb{Q}$-Gorenstein variety such that $(X,0)$ is klt and $-K_{X}$ is big. 
Then there exists $N \in \mathbb{Z}_{>0}$, depending only on $X$, such that every $K_{X}$-MMP terminates after at most $N$ steps. 
\end{cor}

\begin{cor}\label{cor--intro-mmp-CY}
Let $(X,\Delta)$ be a projective klt pair such that $K_{X}+\Delta \equiv 0$. 
Then, for any big $\mathbb{R}$-Cartier divisor $D$ on $X$, every MMP on $D$ terminates. 
\end{cor}

\subsubsection{Sarkisov program}
The Sarkisov program factors birational maps between Mori fiber spaces into elementary birational maps called Sarkisov links. 
In \cite{sarkisov}, Sarkisov introduced a procedure for constructing such links and conjectured that this procedure terminates after finitely many steps. 
In \cite{corti-sarkisov}, Corti proved the termination of this procedure in dimension three. 
Bruno--Matsuki \cite{bruno-matsuki} further developed a logarithmic version of the program. 
In higher dimensions, Hacon--M\textsuperscript{c}Kernan \cite{hacon-mckernan-sarkisov} proved that every birational map between Mori fiber spaces can be decomposed into Sarkisov links. 
More recently, He \cite{he-strong-sarkisov} referred to the termination of the original procedure as the strong Sarkisov program, related it to the termination of certain log flips, and proved it in dimension four. 
Han--Liu--Zhuang \cite{han-liu-zhuang} proved the strong Sarkisov program in dimension five. 

Combining Theorem \ref{thm--full-termination} with \cite[Theorem 1.8]{he-strong-sarkisov}, we obtain the following. 

\begin{cor}\label{cor--intro-strong-sarkisov-program}
The strong Sarkisov program holds in arbitrary dimension.
\end{cor}

\subsubsection{MMP algorithm}

Yasuda \cite{yasuda-mmp-algorithm} constructed an algorithm for a version of the MMP in dimension three over a computable field of characteristic zero admitting a splitting algorithm. 
As explained in \cite{yasuda-mmp-algorithm}, the restriction to dimension three comes from the termination of flips and the abundance conjecture. 
For varieties with big canonical divisor or big anticanonical divisor, our termination results remove the first obstruction, while the second obstruction is removed by the base point free theorem in the former case and does not arise in the latter case. 
We may therefore expect Yasuda's construction to extend to such varieties in arbitrary dimension. 
In fact, we have the following statement. 

\begin{cor}\label{cor--intro-MMP-algorithm}
The MMP algorithm in the sense of Yasuda \cite{yasuda-mmp-algorithm} is valid for normal projective $\mathbb{Q}$-Gorenstein varieties $X$ such that $(X,0)$ is klt and $K_{X}$ or $-K_{X}$ is big. 
\end{cor}

We note that the ground field in \cite{yasuda-mmp-algorithm} is computable and not the field of complex numbers.
Nevertheless, the geometric results needed for the algorithm can be verified after base change to the field of complex numbers.

\subsection{Organization of the paper}
The contents of this paper are as follows.  
In Section~\ref{sec--pre}, we collect definitions and known results used in this paper. 
In Section~\ref{sec--fin-ext-ray-func}, we prove the finiteness of the extremal-ray functions and some related results. 
In Section~\ref{sec--base}, we study the bases of Mori fiber spaces. 
In Section~\ref{sec--ind}, we establish several results used in the inductive argument. 
In Section~\ref{sec--inv-mmp}, we study inverse MMP steps. 
In Section~\ref{sec--prf}, we prove the main results and their applications. 
In Section~\ref{sec--exam}, we discuss the relation between the finiteness of marked Mori fiber spaces and full termination. 

\section*{AI disclosure}
The discussions with ChatGPT described below were conducted using GPT-5.6 Sol and GPT-6, while the Codex run used GPT-6 Astra. 

The author developed a specific formulation of the finiteness problem for marked Mori fiber spaces, which was the initial focus of this project.
The author proposed induction on dimension as a possible approach and suggested examining the results in \cite{han-qi-zhuang}. 
These and other ideas, including the use of the uniform bound on Cartier indices established there, were explored through extensive discussions with ChatGPT.
These discussions led to a uniform perturbation argument along a fixed MMP using this bound. 
A concrete form of what later became extremal-ray functions also emerged from these discussions. 
A working manuscript incorporating both was provided to Codex. 

Building on this working manuscript, Codex produced a draft containing detailed proofs of the finiteness of marked Mori fiber spaces and a more general theorem on marked small birational models (cf.~Theorem \ref{thm--ind-finiteness}) without further input from the author. 
The draft established the finiteness of extremal-ray functions using uniform bounds on Cartier indices (Section~\ref{sec--fin-ext-ray-func}).
The proof of the more general theorem combined induction on dimension using the bases of Mori fiber spaces (Section~\ref{sec--base} and Section~\ref{sec--ind}) with the finiteness of possible inverse MMP steps and a uniform bound on the number of steps of a suitable MMP (Section~\ref{sec--inv-mmp}).
At an earlier stage of the project, ChatGPT proposed using a product with the projective line to deduce full termination in one lower dimension from the finiteness of marked Mori fiber spaces (Section~\ref{sec--exam}).
ChatGPT helped identify and extract arguments from the draft.

The author subsequently made substantial revisions to the arguments in the draft.
In Section~\ref{sec--fin-ext-ray-func}, the author replaced the argument in the Codex draft with a new proof of the finiteness of extremal-ray functions using Shokurov polytopes. 
In Section~\ref{sec--base}, the author revised the argument establishing uniform control of singularities on the bases of Mori fiber spaces. 
In Section~\ref{sec--ind}, the author reformulated the inductive setup and revised the proof using induction on the dimension of these bases. 
In Section~\ref{sec--inv-mmp}, the author replaced the original argument for the finiteness of inverse flips with a shorter proof using extremal-ray functions and the negativity lemma. 
The author also gave a different proof of the uniform bound on the number of steps of a suitable MMP. 
The author also simplified the proofs and their logical dependencies, including the deductions of the main results from the more general finiteness theorem. 
The paper was reorganized to present full termination as a main result.
Throughout these revisions, ChatGPT was used to discuss, refine, and check the mathematical arguments and to assist with their exposition. 

The author independently verified the final proofs and takes full responsibility for the contents of the paper. 

\section*{Acknowledgements}

The author would like to express his deep gratitude to Professor Paolo Cascini, who introduced him to the problem of the finiteness of Mori fiber spaces in 2016, when the author was a Ph.D. student.
The author was partially supported by JSPS KAKENHI Grant Number JP26K06724. 
The author thanks Professors Caucher Birkar and Xiaowei Jiang for helpful comments on the previous version of this paper. 
The author thanks Professors Paolo Cascini, Makoto Enokizono, Osamu Fujino, Kento Fujita, Yoshinori Gongyo, Christopher D. Hacon, Jingjun Han, Yujiro Kawamata, Jihao Liu, James M\textsuperscript{c}Kernan, Joaqu\'{i}n Moraga, Shigefumi Mori, Yusuke Nakamura, Takehiko Yasuda, and Shou Yoshikawa for suggestions and comments.

\section{Preliminaries}\label{sec--pre}

In this section, we define some notions concerning the MMP and we introduce an important result (\cite[Theorem 1.1]{han-qi-zhuang}). 

\subsection{Definitions}
In this subsection we collect definitions.

\begin{defn}[Divisors and morphisms]
Let $X$ be a normal projective variety. 
We use the standard definitions of nef $\mathbb{R}$-divisor, ample $\mathbb{R}$-divisor, semi-ample $\mathbb{R}$-divisor, etc. 
In this paper, we do not assume big $\mathbb{R}$-divisors to be $\mathbb{R}$-Cartier.

Let $f\colon X\to Y$ be a projective morphism of varieties. 
Then $f$ is called a {\em contraction} if it is surjective and has connected fibers. 

A birational map $\phi\colon X \dashrightarrow X'$ of normal projective varieties is called a {\em birational contraction} if $\phi^{-1}$ does not contract any divisor. 
We say that $\phi$ is {\em small} if $\phi$ and $\phi^{-1}$ are birational contractions. 

Let $X$ be a normal projective variety. 
We denote by $N^{1}(X)_{\mathbb{R}}$ the $\mathbb{R}$-vector space of $\mathbb{R}$-Cartier
divisors modulo numerical equivalence. 
We denote the Kleiman--Mori cone by $\overline{\rm NE}(X)$ (see, for example, \cite[Definition 1.17]{kollar-mori}). 
For an $\mathbb{R}$-Cartier divisor $D$ on $X$ and a ray $R\subset\overline{\rm NE}(X)$, we write
$(D\cdot R)>0$ (resp.~$(D\cdot R)=0$, $(D\cdot R)<0$) if $(D\cdot\xi)>0$ (resp.~$(D\cdot\xi)=0$, $(D\cdot\xi)<0$) for any $\xi\in R \setminus\{0\}$. 
\end{defn}

\begin{defn}[Singularities of pairs, {\cite{kollar-mori}, \cite{bchm}, \cite{kollar-mmp}}]
A {\em pair} $(X,\Delta)$ consists of a normal variety $X$ and an effective $\mathbb{R}$-divisor $\Delta$ on $X$ such that $K_{X}+\Delta$ is $\mathbb{R}$-Cartier. 
We use standard definitions of Kawamata log terminal (klt, for short) pair and log canonical (lc, for short) pair. 
\end{defn}

\begin{defn}[Minimal model program]\label{defn--mmp-fullgeneral}
Let $(X,\Delta)$ be a projective lc pair.   

A {\em step of a $(K_{X}+\Delta)$-MMP} is a diagram
 $$
\xymatrix{
(X,\Delta)\ar@{-->}[rr]^-{\phi}\ar[dr]&&(X',\Delta':=\phi_{*}\Delta)\ar[dl]\\
&V
}
$$
consisting of normal projective varieties $X$, $X'$, and $V$ such that
\begin{itemize}
\item
$X \to V$ is a birational morphism and $X' \to V$ is a small birational morphism, 
\item
$-(K_{X}+\Delta)$ is ample over $V$, 
\item
$K_{X'}+\Delta'$ is $\mathbb{R}$-Cartier and ample over $V$, and
\item
the relative Picard number $\rho(X/V)$ is one.  
\end{itemize}
Sometimes we call $X \to V$ a {\em $(K_{X}+\Delta)$-negative extremal contraction}.
\end{defn}

\begin{defn}[Models, cf.~{\cite{bchm}}]
Let $(X,\Delta)$ be a projective lc pair. 
Let $(X',\Delta')$ be a projective pair with a birational contraction $\phi\colon X \dashrightarrow X'$ such that $\Delta'=\phi_{*}\Delta$. 
We say that $(X', \Delta')$ is a {\em log minimal model} of $(X,\Delta)$ if 
\begin{itemize}
\item
$X'$ is $\mathbb{Q}$-factorial, 
\item
$K_{X'}+\Delta'$ is nef, and 
\item
for any prime divisor $P$ on $X$ which is exceptional over $X'$, we have
\begin{equation*}
a(P, X, \Delta) < a(P, X', \Delta').
\end{equation*}
\end{itemize}
We say that $(X',\Delta')$ is a {\em Mori fiber space} of $(X, \Delta)$ if $X'$ is $\mathbb{Q}$-factorial and there is a contraction $X' \to Z$ with ${\rm dim}\,Z<{\rm dim}\,X'$ such that 
\begin{itemize}
\item
$\rho(X'/Z)=1$ and $-(K_{X'}+\Delta')$ is ample over $Z$, and 
\item
for any prime divisor $P$ over $X$, we have
$$a(P,X,\Delta)\leq a(P,X',\Delta')$$
and strict inequality holds if $P$ is a divisor on $X$ and exceptional over $X'$.
\end{itemize}
\end{defn}

\begin{rem}
There are some variants of definitions of log minimal models (cf.~\cite{birkar-flip}). 
In this paper, we only deal with log minimal models and Mori fiber spaces for projective $\mathbb{Q}$-factorial klt pairs. 
Hence, these differences are immaterial for our purposes. 
\end{rem}

\begin{defn}[Marked log minimal model and marked Mori fiber space]
Let $(X,\Delta)$ be a projective lc pair, and let $\phi \colon (X,\Delta) \dashrightarrow (X',\Delta')$ be a birational contraction to a log minimal model (resp.~a Mori fiber space). 
In this paper, we call $(X',\Delta')$ together with $\phi \colon (X,\Delta) \dashrightarrow (X',\Delta')$ a {\em marked log minimal model} (resp.~a {\em marked Mori fiber space}). 

We say that two marked log minimal models (or marked Mori fiber spaces) 
$$\phi_{i} \colon (X,\Delta) \dashrightarrow (X'_{i},\Delta'_{i})\qquad (i=1,2)$$
are {\em isomorphic} if there exists an isomorphism $\psi \colon X'_{1}\to X'_{2}$ such that $\phi_{2}=\psi \circ \phi_{1}$. 
\end{defn}

\subsection{Known and basic results} 

In this subsection we collect known and basic results. 

We start with some lemmas that are well known to experts. 

\begin{lem}[cf.~{\cite[Lemma 2.4]{han-liu-zhuang}}]\label{lem--perturb-coeff}
Let $(X,\Delta)$ be a projective klt pair. 
Then the following statements hold.
\begin{itemize}
\item
If $\Delta$ is big, then there exists an effective ample $\mathbb{R}$-divisor $A$ and an effective $\mathbb{R}$-divisor $G$ on $X$ such that $K_{X}+A+G \sim_{\mathbb{R}} K_{X}+\Delta$ and $(X,A+G)$ is klt. 
\item
If $K_{X}+\Delta$ is big, then there exists an effective ample $\mathbb{R}$-divisor $A'$ and an effective $\mathbb{R}$-divisor $G'$ on $X$ such that $K_{X}+A'+G' \sim_{\mathbb{R}} (1+t)(K_{X}+\Delta)$ for some $t \in \mathbb{R}_{>0}$ and $(X,A'+G')$ is klt. 
\end{itemize}
Furthermore, if $\Delta$ is a $\mathbb{Q}$-divisor on $X$, then we may take these $A$, $G$, $A'$, and $G'$ as $\mathbb{Q}$-divisors on $X$. 
\end{lem}

\begin{proof}
The first assertion immediately follows from \cite[Lemma 3.7.5]{bchm}. 
For the second assertion, write $K_{X}+\Delta \sim_{\mathbb{R}} H+E$ for an effective ample $\mathbb{R}$-divisor $H$ and an effective $\mathbb{R}$-divisor $E$ on $X$, and set $A':=tH$ and $G':=\Delta+ tE$ for some sufficiently small  $t \in \mathbb{R}_{>0}$. 
The same arguments work in the case of rational coefficients.  
\end{proof}

\begin{lem}\label{lem--MFS-basic}
Let $(X,\Delta)$ be a projective $\mathbb{Q}$-factorial klt pair. 
Let $\pi \colon (X,\Delta) \to Z$ be a Mori fiber space, and let $R$ be the $(K_{X}+\Delta)$-negative extremal ray defining $\pi$. 
Then the following properties hold.
\begin{itemize}
\item
$Z$ is $\mathbb{Q}$-factorial. 
\item
For any effective $\mathbb{Q}$-divisor $E$ on $X$, the following conditions are equivalent.
\begin{itemize}
\item
$\pi({\rm Supp}\,E)\subsetneq Z$,
\item
$(E \cdot R)=0$, and
\item
$E=\pi^{*}G_{Z}$ for some effective $\mathbb{Q}$-divisor $G_{Z}$ on $Z$. 
\end{itemize} 
\end{itemize} 
\end{lem}

\begin{proof}
The first assertion is \cite[Corollary 3.18]{kollar-mori}. 
For the second assertion, it is clear that the first property implies the second property and the third property implies the first property. 
Hence, it is enough to prove that $(E \cdot R)=0$ implies $E=\pi^{*}G_{Z}$ for some effective $\mathbb{Q}$-divisor $G_{Z}$ on $Z$. 
By replacing $E$ with some multiple, we may assume that $E$ is Cartier. 
By \cite[Corollary 3.17]{kollar-mori}, there exists a Cartier divisor $D_{Z}$ on $Z$ such that $E \sim \pi^{*}D_{Z}$. 
Then the pullback of rational functions on $Z$ by $\pi$ induces an isomorphism $\pi^{*}\colon \mathcal{O}_{Z}(D_{Z}) \to \pi_{*}\mathcal{O}_{X}(\pi^{*}D_{Z})$. 
Then 
$$H^{0}(Z, \mathcal{O}_{Z}(D_{Z})) \cong H^{0}(X,\mathcal{O}_{X}(\pi^{*}D_{Z}))$$ via $\pi^{*}$. 
Since $E$ is effective, it determines a nonzero element of $H^{0}(X,\mathcal O_X(\pi^{*}D_Z))$. 
By the above isomorphism, we can find an effective Cartier divisor $G_{Z} \sim D_{Z}$ on $Z$ such that $\pi^{*}G_{Z}=E$. 
This finishes the proof. 
\end{proof}

\begin{lem}\label{lem--log-min-model-small}
Let $(X,\Delta)$ be a projective klt pair, and let $(X',\Delta')$ be another projective klt pair. 
Suppose that there exists a small birational map $\phi \colon X \dashrightarrow X'$ such that $\Delta' = \phi_{*}\Delta$. 
Then a projective klt pair $(X'',\Delta'')$ is a log minimal model of $(X,\Delta)$ if and only if $(X'',\Delta'')$ is a log minimal model of $(X',\Delta')$. 
\end{lem}

\begin{proof}
This immediately follows from the definition of log minimal models. 
\end{proof}

\begin{lem}\label{lem--finite-log-min-model-1}
Let $(X,\Delta)$ be a projective klt pair such that $K_{X}+\Delta$ or $\Delta$ is big. 
Let $\mathfrak{V}$ be a set of small birational maps
$\phi\colon (X,\Delta)\dashrightarrow (X',\Delta')$ to projective klt pairs $(X',\Delta')$ such that $\Delta'=\phi_*\Delta$.
Then the set
$$\left\{\tau \colon (X,\Delta) \dashrightarrow (Y,\Gamma)\,\middle|\,\begin{array}{l}
\text{$\tau=\psi \circ \phi$ where $\phi \colon (X,\Delta) \dashrightarrow (X',\Delta') \in \mathfrak{V}$ and}\\
\text{$\psi \colon (X',\Delta') \dashrightarrow (Y,\Gamma)$ is a birational contraction}\\
\text{to a log minimal model $(Y,\Gamma)$}\end{array}\right\}$$
is a finite set up to isomorphism compatible with the markings. 
\end{lem}

\begin{proof}
This follows from Lemma \ref{lem--perturb-coeff}, Lemma \ref{lem--log-min-model-small}, and \cite[Theorem E]{bchm}. 
\end{proof}

The following result is a key ingredient for the proof of  the main results of this paper. 

\begin{thm}[{\cite[Theorem 1.1]{han-qi-zhuang}}]\label{thm--mmp-Cartierindex-bound}
Let $(X,\Delta)$ be a projective $\mathbb{Q}$-factorial klt pair such that $K_{X}+\Delta$ or $\Delta$ is big. 
Then, there exists $\ell \in \mathbb{Z}_{>0}$, depending only on $(X,\Delta)$, such that for any finite sequence of steps of a $(K_{X}+\Delta)$-MMP $(X,\Delta) \dashrightarrow (X',\Delta')$ and any Weil divisor $D'$ on $X'$, $\ell D'$ is Cartier. 
\end{thm}

We close this section with the definition of linear functions used in this paper, called ``extremal-ray functions'' in the introduction.

\begin{defn}\label{defn--intnum-func}
Let $(X,\Delta)$ be a projective lc pair. 
For any finite sequence of steps of a $(K_{X}+\Delta)$-MMP $(X,\Delta)\dashrightarrow (Y,\Delta_{Y})$ and any $(K_{Y}+\Delta_{Y})$-negative extremal ray $R$ of $\overline{\rm NE}(Y)$, we define 
$$\zeta_{Y,R} \colon N^{1}(X)_{\mathbb{R}}\longrightarrow \mathbb{R}$$
as follows: 
We fix a very ample Cartier divisor $H$ on $Y$, and then we pick an associated minimal curve $C_{R}$ lying on $R$, that is, a curve $C_{R}$ on $Y$ such that $C_{R} \in R$ and
$$(H \cdot C_{R}) ={\rm min}\{(H\cdot C)\,|\,\text{$C$ is a curve lying on $R$}\}.$$
By \cite[Remark 6.1]{hashizumehu}, for any $\mathbb{R}$-Cartier divisor $D$ on $X$, the birational transform $D_{Y}$ on $Y$ is $\mathbb{R}$-Cartier. 
We define 
$$\zeta_{Y,R}(D):=(D_{Y}\cdot C_{R}).$$
Note that this is well-defined. Indeed, by \cite[Remark 6.1]{hashizumehu} and the negativity lemma,
the map $X\dashrightarrow Y$ induces a linear map
$N^{1}(X)_{\mathbb{R}}\to N^{1}(Y)_{\mathbb{R}}$. 
We can also check that this does not depend on the very ample Cartier divisor $H$ and the minimal curve $C_{R}$. 
\end{defn}

\begin{rem}\label{rem--basic-intnum-func}
With notation as in Definition \ref{defn--intnum-func}, the following properties hold.
\begin{itemize}
\item
If $D_{Y}$ is Cartier, then $\zeta_{Y,R}(D) \in \mathbb{Z}$. 
\item
If $D_{Y} \equiv K_{Y}+\Gamma$ for some lc pair $(Y,\Gamma)$ and $R$ is a $(K_{Y}+\Gamma)$-negative extremal ray, we have $-2 \cdot {\rm dim}\,X \leq \zeta_{Y,R}(D) <0$. 
\end{itemize}
\end{rem}

\section{Finiteness of extremal-ray functions}\label{sec--fin-ext-ray-func}

In this section we prove the finiteness of extremal-ray functions (Definition \ref{defn--intnum-func}) for projective $\mathbb{Q}$-factorial klt pairs of general type or with big boundaries.

\begin{prop}\label{prop--mmp-perturb-boundary}
Let $(X,\Delta)$ be a projective $\mathbb{Q}$-factorial klt pair such that $K_{X}+\Delta$ or $\Delta$ is big. 
Let $B$ and $G$ be effective $\mathbb{R}$-divisors on $X$ such that $(X,B+G)$ is klt and $\Delta \equiv B+G$.
Then there exists a positive real number $\epsilon$ such that 
for any sequence of steps of a $(K_{X}+\Delta)$-MMP $\phi \colon (X,\Delta) \dashrightarrow (X',\Delta')$ and any $(K_{X'}+\Delta')$-negative extremal ray $R'$ of $\overline{\rm NE}(X')$, we have 
$$(K_{X'}+(1+t)\phi_{*}B+\phi_{*}G)\cdot R'<0$$ for all $t \in [-\epsilon, \epsilon]$. 
In particular, any $(K_{X}+\Delta)$-MMP is a $(K_{X}+(1+t)B+G)$-MMP for every $t \in [-\epsilon, \epsilon]$. 
\end{prop}

\begin{proof}
We may assume $\Delta=B+G$ since any sequence of steps of a $(K_{X}+\Delta)$-MMP is also a $(K_{X}+B+G)$-MMP. 
We fix $\ell \in \mathbb{Z}_{>0}$ as in Theorem \ref{thm--mmp-Cartierindex-bound} throughout this proof.

By applying \cite[Theorem 5.6]{han-liu-shokurov}, we can find positive real numbers $r_{1},\,\cdots,\, r_{p}$ and $\mathbb{Q}$-divisors $\Delta^{(1)},\,\cdots,\, \Delta^{(p)}$ on $X$ such that 
\begin{itemize}
\item
$\sum_{i=1}^{p}r_{i}=1$ and $\sum_{i=1}^{p}r_{i}\Delta^{(i)}=\Delta=B+G$, and
\item
$(X,\Delta^{(i)})$ is lc for every $i$. 
\end{itemize}
Fix a positive integer $\ell'$ such that all $\ell' \Delta^{(i)}$ are Weil divisors on $X$. 

We fix a very ample Cartier divisor $H$ on $X$. 
For any $(K_{X}+B+G)$-negative extremal ray $R \subset \overline{\rm NE}(X)$, there exists a minimal curve $C_{R}$ lying on $R$. 
By the length of extremal rays (\cite[Theorem 3.7]{kollar-mori}, \cite[Theorem 3.8.1]{bchm}) and the minimality of $C_{R}$, we have $(K_{X}+\Delta^{(i)})\cdot C_{R} \geq -2 \cdot {\rm dim}\,X$ for every $1 \leq i \leq p$. Since all $\ell\ell'(K_{X}+\Delta^{(i)})$ are Cartier, we have 
$$(K_{X}+B+G)\cdot C_{R} \in \left \{\sum_{i=1}^{p}r_{i}\alpha_{i} \,\,\middle|\,\,  \text{$\alpha_{i}\in \frac{1}{\ell\ell'}\mathbb{Z}$ \,and\, $\alpha_{i} \geq -2 \cdot {\rm dim}\,X$} \right\}\cap (-\infty, 0).$$
Hence, there exists a positive real number $\alpha$, depending only on $r_{1},\,\cdots,\, r_{p}$, $\ell$, $\ell'$, and ${\rm dim}\,X$, such that $(K_{X}+B+G)\cdot C_{R} \leq -\alpha$. 
 
Since $(X,G)$ is klt, the minimality of $C_{R}$ implies $(K_{X}+G)\cdot C_{R}\geq -2 \cdot {\rm dim}\,X$. 
We fix $\delta \in \mathbb{R}_{>0}$ such that $(X,(1+\delta)B+G)$ is klt, and we put 
$$\epsilon':={\rm min}\left\{\delta, \, \frac{\alpha}{4\cdot {\rm dim}\,X}\right\}.$$  
Note that $\epsilon'$ depends only on $\delta$, $\alpha$, and ${\rm dim}\,X$. 
For any $t \in[0,\epsilon']$, the definition of $\epsilon'$ shows that $(X,(1+t)B+G)$ is a klt pair and we have
\begin{equation*}
\begin{split}
-2 \cdot {\rm dim}\,X & \leq (K_{X}+(1+t)B+G)\cdot C_{R}\\
& =  (1+t)(K_{X}+B+G)\cdot C_{R}-t(K_{X}+G)\cdot C_{R}\\
& \leq -(1+t)\alpha+2t\cdot{\rm dim}\,X\\
& \leq -(1+t)\alpha+2\cdot{\rm dim}\,X \cdot\frac{\alpha}{4\cdot {\rm dim}\,X}=-\left(\frac{1}{2}+t\right)\alpha\\
& < 0.
\end{split}
\end{equation*}
We put 
$$\epsilon'':=\frac{\alpha \epsilon'}{4\cdot {\rm dim}\,X}.$$
Then $\epsilon'' \leq \epsilon'$, and $\epsilon''$ also depends only on $\delta$, $\alpha$, and ${\rm dim}\,X$. 
For any $u \in[0, \epsilon'']$, using 
$$-2 \cdot {\rm dim}\,X \leq (K_{X}+(1+\epsilon')B+G)\cdot C_{R} \quad {\rm and} \quad (K_{X}+B+G)\cdot C_{R} \leq -\alpha$$
we have
\begin{equation*}
\begin{split}
& \mathrel{\phantom{=}} (K_{X}+(1-u)B+G)\cdot C_{R}\\
& =\frac{\epsilon'+u}{\epsilon'}(K_{X}+B+G)\cdot C_{R}-\frac{u}{\epsilon'}(K_{X}+(1+\epsilon')B+G)\cdot C_{R}\\
& \leq -\frac{\epsilon'+u}{\epsilon'}\alpha+2\frac{u}{\epsilon'}\cdot{\rm dim}\,X\\
& \leq -\frac{\epsilon'+u}{\epsilon'}\alpha+\frac{2}{\epsilon'}\cdot{\rm dim}\,X \cdot\frac{\alpha \epsilon'}{4\cdot {\rm dim}\,X}=-\left(\frac{1}{2}+\frac{u}{\epsilon'}\right)\alpha\\
& < 0.
\end{split}
\end{equation*}
Therefore, we see that $ (K_{X}+(1+t)B+G)\cdot C_{R}<0$ for any $t \in[-\epsilon'',\epsilon']$. 

Let 
$$(X,B+G) \dashrightarrow (X_{1},B_{1}+G_{1})\dashrightarrow \cdots \dashrightarrow (X',B'+G')$$
be a sequence of steps of a $(K_{X}+\Delta)$-MMP, and let $R'$ be a $(K_{X'}+\Delta')$-negative extremal ray of $\overline{\rm NE}(X')$, where $\Delta'$ is the birational transform of $\Delta$ on $X'$. 
By induction on the number of steps, we may assume that
$X \dashrightarrow X'$ defines a $(K_{X}+(1+t)B+G)$-MMP for all $t \in [-\epsilon'', \epsilon']$. 
For each $1 \leq i \leq p$, let $\Delta'^{(i)}$ be the birational transform of $\Delta^{(i)}$ on $X'$. 
Then we can apply the previous discussion for $(X',\Delta'=B'+G')$, $\Delta'^{(1)},\,\cdots,\, \Delta'^{(p)}$, $r_{1},\,\cdots,\, r_{p}$, $\ell$, and $\ell'$. 
Indeed, since ${\rm dim}\,X={\rm dim}\,X'$ and every coefficient of $\Delta'$ belongs to the set of the coefficients of $\Delta$, our choice of $r_{1},\,\cdots,\, r_{p}$ and \cite[Theorem 5.6]{han-liu-shokurov} imply that 
\begin{itemize}
\item
$\sum_{i=1}^{p}r_{i}=1$ and $\sum_{i=1}^{p}r_{i}\Delta'^{(i)}=\Delta'$,  
\item
$(X',\Delta'^{(i)})$ is lc for every $i$, and 
\item
every $\ell'\Delta'^{(i)}$ is a Weil divisor on $X'$.  
\end{itemize}
Note that the second condition follows from the uniformity in \cite[Theorem 5.6]{han-liu-shokurov}, namely, that the corresponding neighborhood is independent of the given pair.
Fix a very ample Cartier divisor on $X'$ and an associated minimal curve $C_{R'}$ lying on $R'$. 
Since every $\ell\ell'(K_{X'}+\Delta'^{(i)})$ is Cartier, as in the previous argument, we have 
$$(K_{X'}+B'+G')\cdot C_{R'} \leq -\alpha$$
for the positive real number $\alpha$ that was previously defined. 
Now $(X',(1+\epsilon')B'+G')$ is klt since $X \dashrightarrow X'$ defines a $(K_{X}+(1+\epsilon')B+G)$-MMP. 
By the minimality of $C_{R'}$ and using the previous argument after we replace $\delta$ with $\epsilon'$, we have $(K_{X'}+(1+t)B'+G')\cdot C_{R'}<0$ for any $t \in[-\epsilon'',\epsilon']$. 

Finally, defining $\epsilon:=\epsilon''={\rm min}\{\epsilon'', \epsilon'\}$, we complete the proof of Proposition \ref{prop--mmp-perturb-boundary}. 
\end{proof}

\begin{cor}\label{cor--mmp-boundary-polytope}
Let $(X,\Delta)$ be a projective $\mathbb{Q}$-factorial klt pair such that $K_{X}+\Delta$ or $\Delta$ is big. 
Then there exist finitely many effective $\mathbb{Q}$-divisors $B_{1},\,\cdots ,\,B_{p}$ on $X$ satisfying the following. 
Let $\mathcal{C} \subset N^{1}(X)_{\mathbb{R}}$ be the rational polytope spanned by $B_{1},\,\cdots ,\,B_{p}$. 
Then
\begin{itemize}
\item
$(X,B_{i})$ is klt for every $1 \leq i \leq p$, 
\item
${\rm dim}\, \mathcal{C} =\rho(X)$, 
\item
for any klt pair $(X,B)$ such that $B \in \mathcal{C}$ as a numerical class, any sequence of steps of a $(K_{X}+\Delta)$-MMP $\phi \colon (X,\Delta) \dashrightarrow (X',\Delta')$, and any $(K_{X'}+\Delta')$-negative extremal ray $R'$ of $\overline{\rm NE}(X')$, we have 
$$(K_{X'}+\phi_{*}B)\cdot R'<0,$$ 
in particular, any $(K_{X}+\Delta)$-MMP is also a $(K_{X}+B)$-MMP, and 
\item
if $\Delta$ is big, then $\Delta$ is contained in the interior ${\rm int}(\mathcal{C})$ of $\mathcal{C}$ as a numerical class. 
\end{itemize}
\end{cor}

\begin{proof}
We first assume that $\Delta$ is big. 
Replacing $\Delta$ by a suitable $\mathbb{R}$-linearly equivalent $\mathbb{R}$-divisor, we may write $$\Delta=\sum_{i=1}^{m}r_{i}A_{i}+G,$$ where $r_{i}$ are positive real numbers, $A_{i}$ are effective ample $\mathbb{Q}$-divisors on $X$ spanning $N^{1}(X)_{\mathbb{R}}$, and $G$ is an effective $\mathbb{Q}$-divisor on $X$. 
By Proposition \ref{prop--mmp-perturb-boundary}, for each $1 \leq i \leq m$, there exists $\epsilon_{i} \in \mathbb{R}_{>0}$ satisfying the property stated in Proposition \ref{prop--mmp-perturb-boundary}. 
Let $\mathcal{C}_{0} \subset N^{1}(X)_{\mathbb{R}}$ be the polytope spanned by $\Delta \pm \epsilon_{1}A_{1},\,\cdots ,\,\Delta \pm \epsilon_{m}A_{m}$. 
Then we have ${\rm dim}\, \mathcal{C}_{0} =\rho(X)$ and $\Delta \in {\rm int}(\mathcal{C}_{0})$ by construction.  
Moreover, for every klt pair $(X,B)$ such that $B \in \mathcal{C}_{0}$ as a numerical class, $K_{X}+B$ is numerically equivalent to a convex $\mathbb{R}_{\geq 0}$-linear combination of $K_{X}+\Delta \pm \epsilon_{1}A_{1},\,\cdots ,\,K_{X}+\Delta \pm \epsilon_{m}A_{m}$. 
Hence, the choice of $\epsilon_{1},\,\cdots,\,\epsilon_{m}$ implies that for any $(K_{X}+\Delta)$-MMP $\phi \colon (X,\Delta) \dashrightarrow (X',\Delta')$ and any $(K_{X'}+\Delta')$-negative extremal ray $R'$ of $\overline{\rm NE}(X')$, we have 
$$(K_{X'}+\phi_{*}B)\cdot R'<0.$$ 
Take a rational polytope $\mathcal{T}$ in the $\mathbb{R}$-vector space spanned by components of $\Delta$ so that any element of $\mathcal{T}$ satisfies the first condition of Corollary \ref{cor--mmp-boundary-polytope}, the image $\mathcal{C}$ of $\mathcal{T}$ in $N^{1}(X)_{\mathbb{R}}$ satisfies the second and the fourth conditions of Corollary \ref{cor--mmp-boundary-polytope}, and $\mathcal{C} \subset \mathcal{C}_{0}$. 
The inclusion $\mathcal{C} \subset \mathcal{C}_{0}$ implies the third condition of Corollary \ref{cor--mmp-boundary-polytope}. 
Then the vertices $B_{1},\,\cdots,\,B_{p}$ of $\mathcal{T}$ satisfy all the conditions of Corollary \ref{cor--mmp-boundary-polytope}. 

Next, we assume that $K_{X}+\Delta$ is big. 
By the second assertion of Lemma \ref{lem--perturb-coeff}, replacing $\Delta$ appropriately, we can assume that $\Delta$ is big. 
Then the above argument implies the existence of $B_{1},\,\cdots,\,B_{p}$ as in Corollary \ref{cor--mmp-boundary-polytope}. 
\end{proof}

\begin{thm}\label{thm--finite-intnum-func}
Let $(X,\Delta)$ be a projective $\mathbb{Q}$-factorial klt pair such that $K_{X}+\Delta$ or $\Delta$ is big. 
Then the set of functions in Definition \ref{defn--intnum-func}
\begin{equation*}
\left\{\zeta_{X',R'}\,\,\middle|
\begin{array}{l}\text{$(X,\Delta) \dashrightarrow (X',\Delta')$ is a $(K_{X}+\Delta)$-MMP,}\\\text{$R'$ is a $(K_{X'}+\Delta')$-negative extremal ray}\end{array}
\right\}
\end{equation*}
is a finite set. 
\end{thm}

\begin{proof}
By Corollary \ref{cor--mmp-boundary-polytope}, there exist finitely many effective $\mathbb{Q}$-divisors $B_{1},\,\cdots,\, B_{p}$ on $X$ such that if we put $\mathcal{C} \subset N^{1}(X)_{\mathbb{R}}$ as the rational polytope spanned by $B_{1},\,\cdots,\, B_{p}$, then 
\begin{itemize}
\item
$(X,B_{i})$ is klt for every $1 \leq i \leq p$, 
\item
${\rm dim}\, \mathcal{C} =\rho(X)$, and
\item
for any klt pair $(X,B)$ such that $B \in \mathcal{C}$ as a numerical class, any sequence of steps of a $(K_{X}+\Delta)$-MMP $\phi \colon (X,\Delta) \dashrightarrow (X',\Delta')$, and any $(K_{X'}+\Delta')$-negative extremal ray $R'$ of $\overline{\rm NE}(X')$, we have $(K_{X'}+\phi_{*}B)\cdot R'<0$. 
\end{itemize}
By the second property and the $\mathbb{Q}$-factoriality of $X$, for any $\mathbb{R}$-divisor $D$ on $X$, we can find real numbers $r_{1},\,\cdots,\, r_{p}$, some of which may be negative, such that $\sum_{i=1}^{p}r_{i}=1$ and $D-K_{X} \equiv \sum_{i=1}^{p}r_{i}B_{i}$. 
Then $D \equiv \sum_{i=1}^{p}r_{i}(K_{X}+B_{i})$. 
By Definition \ref{defn--intnum-func}, we have
$$\zeta_{X',R'}(D)=\sum_{i=1}^{p}r_{i}\zeta_{X',R'}(K_{X}+B_{i}).$$
Hence, $\zeta_{X',R'}(D)$ is determined by $\zeta_{X',R'}(K_{X}+B_{i})$ $(1 \leq i \leq p)$. 
From this fact, we only have to show that for each $1 \leq i \leq p$, the set
\begin{equation*}
\mathcal{S}_{i}:=\left\{\zeta_{X',R'}(K_{X}+B_{i})\,\,\middle|
\begin{array}{l}\text{$(X,\Delta) \dashrightarrow (X',\Delta')$ is a $(K_{X}+\Delta)$-MMP,}\\\text{$R'$ is a $(K_{X'}+\Delta')$-negative extremal ray}\end{array}
\right\}
\end{equation*}
is a finite set. 

Let $\ell$ be the positive integer as in Theorem \ref{thm--mmp-Cartierindex-bound}, and let $\ell'$ be a positive integer such that all $\ell' B_{i}$ are Weil divisors on $X$. 
By the third property stated above, for every $1 \leq i \leq p$, any $(K_{X}+\Delta)$-MMP $\phi \colon (X,\Delta) \dashrightarrow (X',\Delta')$ is also a $(K_{X}+B_{i})$-MMP, $R'$ is a $(K_{X'}+\phi_{*}B_{i})$-negative extremal ray, and the divisor $\ell\ell'(K_{X'}+\phi_{*}B_{i})$ is Cartier. 
By Remark \ref{rem--basic-intnum-func}, we see that 
$$\zeta_{X',R'}(K_{X}+B_{i}) \in [-2 \cdot {\rm dim}\,X,0) \cap \frac{1}{\ell\ell'}\mathbb{Z}$$
and the right hand side is a finite set and it does not depend on $X'$ and $R'$. 
Therefore, $\mathcal{S}_{i}$ is a finite set, and the set of $p$-tuples $\big\{\big(\zeta_{X',R'}(K_{X}+B_{1}),\,\cdots,\,\zeta_{X',R'}(K_{X}+B_{p})\big)\big\}_{X',R'}$ is also a finite set. 
Since $\zeta_{X',R'}$ is determined by $\zeta_{X',R'}(K_{X}+B_{i})$ $(1 \leq i \leq p)$, which was discussed above, we see that $\{\zeta_{X',R'}\}_{X',R'}$ is a finite set.
Theorem \ref{thm--finite-intnum-func} follows from this fact. 
\end{proof}

\begin{thm}\label{thm--finite-div-cont-mmp}
Let $(X,\Delta)$ be a projective $\mathbb{Q}$-factorial klt pair such that $K_{X}+\Delta$ or $\Delta$ is big. 
Then there exists a finite set $\mathcal{Q}$ of prime divisors on $X$ satisfying the following. 
Let $(X,\Delta) \dashrightarrow (X',\Delta')$ be an arbitrary sequence of steps of a $(K_{X}+\Delta)$-MMP and $P$ a prime divisor on $X$ contracted by $X \dashrightarrow X'$. 
Then $P \in \mathcal{Q}$. 
\end{thm}

\begin{proof}
Suppose for contradiction that there exists an infinite set of prime divisors $\{P_{i}\}_{i=1}^{\infty}$ with corresponding sequences of $(K_{X}+\Delta)$-MMP
$$(X,\Delta) \dashrightarrow (X^{(i)},\Delta^{(i)})$$
so that the final step $\tilde{X}^{(i)} \dashrightarrow X^{(i)}$ of each $X \dashrightarrow X^{(i)}$ is the divisorial contraction of $P_{i}$. 
For each $i$, let $R_{i} \subset \overline{\rm NE}(\tilde{X}^{(i)})$ be the extremal ray that defines $\tilde{X}^{(i)} \to X^{(i)}$.
Then the function $\zeta_{\tilde{X}^{(i)},R_{i}}$ defined in  Definition \ref{defn--intnum-func} satisfies $\zeta_{\tilde{X}^{(i)},R_{i}}(P_{i})<0$ and $\zeta_{\tilde{X}^{(i)},R_{i}}(Q)\geq 0$ for any prime divisor $Q \neq P_{i}$ on $X$. 
This shows that $\{\zeta_{\tilde{X}^{(i)},R_{i}}\}_{i=1}^{\infty}$ is an infinite set. 
However, it contradicts Theorem \ref{thm--finite-intnum-func}. 
Hence, Theorem \ref{thm--finite-div-cont-mmp} holds. 
\end{proof}

\begin{thm}\label{thm--two-mfs-small}
Let $(X,\Delta)$ be a projective $\mathbb{Q}$-factorial klt pair such that $K_{X}+\Delta$ is not pseudo-effective. 
Let 
$$(X,\Delta) \dashrightarrow (X',\Delta') \overset{\pi'}{\longrightarrow} Z' \qquad {\rm and} \qquad (X,\Delta) \dashrightarrow (X'',\Delta'') \overset{\pi''}{\longrightarrow} Z''$$ be two sequences of steps of a $(K_{X}+\Delta)$-MMP with Mori fiber spaces $\pi' \colon (X',\Delta') \to Z'$ and $\pi''\colon (X'',\Delta'') \to Z''$, respectively. 
Let $R' \subset \overline{\rm NE}(X')$ and $R'' \subset \overline{\rm NE}(X'')$ be extremal rays that define $\pi'$ and $\pi''$, respectively. 
Suppose that 
\begin{itemize}
\item
the induced birational map $\phi \colon (X',\Delta') \dashrightarrow (X'',\Delta'')$ is small, and
\item
$\zeta_{X',R'}=c\zeta_{X'',R''}$ as functions (see Definition \ref{defn--intnum-func}) for some $c \in \mathbb{R}_{>0}$. 
\end{itemize}
Then there exists a small birational map $\psi \colon Z' \dashrightarrow Z''$ such that $\pi'' \circ \phi=\psi \circ \pi'$. 
\end{thm}

\begin{proof}
Let $D'$ be an arbitrary $\mathbb{Q}$-divisor on $X'$, and let $D$ be a $\mathbb{Q}$-divisor on $X$ whose birational transform on $X'$ is $D'$. 
Then $\zeta_{X',R'}(D)=0$ if and only if $(D'\cdot R')=0$. 
Since $\pi' \colon X' \to Z'$ is a Mori fiber space, $(D'\cdot R')=0$ if and only if $D' \sim_{\mathbb{Q}} \pi'^{*}L_{Z'}$ for some $\mathbb{Q}$-divisor $L_{Z'}$ on $Z'$, and the same equivalence holds for $\phi_{*}D'$. 
Since $\zeta_{X',R'}=c\zeta_{X'',R''}$ with $c \in \mathbb{R}_{>0}$, it follows that $D' \sim_{\mathbb{Q}} \pi'^{*}L_{Z'}$ for some $\mathbb{Q}$-divisor $L_{Z'}$ on $Z'$ if and only if $\phi_{*}D' \sim_{\mathbb{Q}} \pi''^{*}L_{Z''}$ for some $\mathbb{Q}$-divisor $L_{Z''}$ on $Z''$. 

Let $f' \colon W \to X'$ and $f'' \colon W \to X''$ be a common resolution of $\phi$. 
Now we have the following diagram
$$
\xymatrix{
&W\ar[dl]_{f'}\ar[dr]^{f''}\\
X'\ar@{-->}[rr]^-{\phi}\ar[d]_{\pi'}&&X''\ar[d]^{\pi''}\\
Z'&&Z''.
}
$$
Let $A_{Z'}$ be an ample $\mathbb{Q}$-divisor on $Z'$. 
We put $H':=\pi'^{*}A_{Z'}$ and $H'':=\phi_{*}H'$. 
By the negativity lemma, there exists an effective $f''$-exceptional $\mathbb{Q}$-divisor $E_{W}$, which is also $f'$-exceptional since $\phi$ is small, on $W$ such that
$$E_{W}+f'^{*}H' =f''^{*}H''.$$
By the above argument, we have $H''\sim_{\mathbb{Q}} \pi''^{*}A_{Z''}$ for some $\mathbb{Q}$-divisor $A_{Z''}$ on $Z''$. 
Since $H'$ is semi-ample and $E_{W}$ is $f'$-exceptional, we have
$${\rm Supp}\,E_{W}={\rm Bs}|E_{W}+f'^{*}H'|_{\mathbb{Q}} = {\rm Bs}|f''^{*}H''|_{\mathbb{Q}}=(\pi'' \circ f'')^{-1}({\rm Bs}|A_{Z''}|_{\mathbb{Q}}).$$
Therefore, $E_{W}$ is vertical over $Z''$. 
Let $F$ be a general fiber of $\pi'' \circ f'' \colon W \to Z''$. 
Then we have $E_{W}|_{F}=0$ and $f''^{*}H''|_{F}\sim_{\mathbb{Q}}0$. 
Let $F \overset{\mu}{\longrightarrow} T \to Z'$ be the Stein factorization of the induced morphism $(\pi' \circ f')|_{F} \colon F \to Z'$, and let $A_{T}$ be the pullback of $A_{Z'}$ to $T$. 
Then 
$$\mu^{*}A_{T}=(f'^{*}\pi'^{*}A_{Z'})|_{F}\sim_{\mathbb{Q}}(E_{W}+f'^{*}H')|_{F} =f''^{*}H''|_{F}\sim_{\mathbb{Q}}0.$$
Since $A_{T}$ is ample, we have ${\rm dim}\,T=0$. 
Therefore, $(\pi' \circ f')(F)$ is a point. 

Let $h \colon W \to V$ be the Stein factorization of the induced morphism $W \to Z' \times Z''$, and let $q'\colon V \to Z'$ and $q''\colon V \to Z''$ be the structure morphisms. 
For any general point $z'' \in Z''$, the above argument shows $q'(q''^{-1}(z''))$ is a point. 
Hence, the construction of $W \overset{h}{\longrightarrow} V \overset{q''}{\longrightarrow} Z''$ implies that $q''^{-1}(z'')$ is a point for any general point $z''$. 
Therefore, $q'' \colon V \to Z''$ is birational. 
By a similar argument, we see that $q' \colon V \to Z'$ is birational. 
This shows that $\psi :=q'' \circ q'^{-1} \colon Z' \dashrightarrow Z''$ is a birational map and $\pi'' \circ \phi=\psi \circ \pi'$. 

We show that $\psi$ is small. 
Let $P'$ be a prime divisor on $Z'$, and let $Q'$ be a prime divisor on $X'$ such that $\pi'(Q')=P'$. 
By Lemma \ref{lem--MFS-basic}, we have $Q'=\alpha\pi'^{*}P'$ for some $\alpha \in \mathbb{Q}_{>0}$. 
Then $Q'':=\phi_{*}Q'$ is a prime divisor on $X''$, and Lemma \ref{lem--MFS-basic} implies that we have $Q''=\beta \pi''^{*}P''$ for some $\beta \in \mathbb{Q}_{>0}$ and prime divisor $P''$ on $Z''$. 
Let $\eta_{P'}$ (resp.~$\eta_{Q'}$, $\eta_{Q''}$, $\eta_{P''}$) be the generic point of $P'$ (resp.~$Q'$, $Q''$, $P''$). 
Then we have
$$\psi(\eta_{P'})=\psi(\pi'(\eta_{Q'}))=\pi''(\phi(\eta_{Q'}))=\pi''(\eta_{Q''})=\eta_{P''}.$$
This shows that $\psi_{*}P'=P'' \neq 0$, and therefore $\psi$ does not contract any divisor. 
By a similar argument, we see that $\psi^{-1}$ does not contract any divisors. 
Thus, $\psi$ is small. 
This completes the proof.  
\end{proof}

\section{Bases of Mori fiber spaces}\label{sec--base}

In this section we study the bases of Mori fiber spaces for projective $\mathbb{Q}$-factorial klt pairs with big boundaries.   

\begin{setup}\label{setup--base-mfs}
Throughout this section, we fix 
\begin{itemize}
\item
a projective $\mathbb{Q}$-factorial klt pair $(X,\Delta)$ such that $\Delta$ is big and $K_{X}+\Delta$ is not pseudo-effective,
\item
a sequence of steps of a $(K_{X}+\Delta)$-MMP 
$$(X,\Delta) \overset{\sigma_{0}}{\dashrightarrow} (X_{0},\Delta_{0}) \overset{\pi_{0}}{\longrightarrow}Z_{0}$$ 
terminating with a Mori fiber space $\pi_{0} \colon (X_{0},\Delta_{0}) \to Z_{0}$, and 
\item
a $(K_{X_{0}}+\Delta_{0})$-negative extremal ray $R_{0}\subset \overline{\rm NE}(X_{0})$ defining $\pi_{0}$. 
\end{itemize}
We define $\mathfrak{M}_{X}$ to be the set of pairs
$$\Bigl((X,\Delta) \overset{\sigma'}{\dashrightarrow} (X',\Delta') \overset{\pi'}{\longrightarrow}Z', R'\Bigr),$$
where $(X,\Delta) \overset{\sigma'}{\dashrightarrow} (X',\Delta') \overset{\pi'}{\longrightarrow}Z'$ is a sequence of steps of a $(K_{X}+\Delta)$-MMP terminating with a Mori fiber space $\pi' \colon (X',\Delta') \to Z'$ and  $R'$ is the $(K_{X'}+\Delta')$-negative extremal ray defining $\pi'$, such that
\begin{itemize}
\item
the induced birational map $\phi \colon X_{0} \dashrightarrow X'$ is small, and 
\item
$\zeta_{X_{0},R_{0}}=c\zeta_{X',R'}$ for some $c \in \mathbb{R}_{>0}$.
\end{itemize} 
We note that $Z_{0}$ and every $Z'$ appearing above are projective $\mathbb{Q}$-factorial varieties.
We also note that there exists a small birational map $\psi \colon Z_{0} \dashrightarrow Z'$ such that $\pi' \circ \phi=\psi \circ \pi_{0}$ by Theorem \ref{thm--two-mfs-small}. 
\end{setup}

\begin{lem}\label{lem--base-klt-threshold}
With notation as in Setup \ref{setup--base-mfs}, let $L_{Z_{0}}$ be an effective $\mathbb{R}$-divisor on $Z_{0}$. 
Then there exists $\alpha \in \mathbb{Q}_{>0}$, depending only on $\sigma_{0}$, $\pi_{0}$, and $L_{Z_{0}}$, satisfying the following. 
Pick any element 
$$\Bigl((X,\Delta) \overset{\sigma'}{\dashrightarrow} (X',\Delta') \overset{\pi'}{\longrightarrow}Z', R'\Bigr)$$ 
of $\mathfrak{M}_{X}$ with the induced small birational map $\psi \colon Z_{0} \dashrightarrow Z'$ such that $\pi' \circ \phi=\psi \circ \pi_{0}$ as in Theorem \ref{thm--two-mfs-small}. 
Then, for any $G_{Z_{0}} \in |L_{Z_{0}}|_{\mathbb{R}}$, the pair $(Z',2\alpha \psi_{*}G_{Z_{0}})$ is klt. 
\end{lem}

\begin{proof}
Let $f \colon W \to X$ and $f_{0} \colon W \to X_{0}$ be a common resolution of $\sigma_{0} \colon X \dashrightarrow X_{0}$. 
We put $L:=f_{*}f_{0}^{*}\pi_{0}^{*}L_{Z_{0}}$. 
Since $X$ is $\mathbb{Q}$-factorial, $L$ is an effective $\mathbb{R}$-Cartier divisor on $X$. 
Fix a very ample Cartier divisor $A$ on $X$ such that $A-\Delta$ and $A-L$ are pseudo-effective. 
By applying \cite[Theorem 1.8]{birkar-bab} to $(X,\Delta)$, $A$, and $L$, we can find $\alpha \in \mathbb{Q}_{>0}$ such that we have ${\rm lct}(X,\Delta;|L|_{\mathbb{R}}) \geq 3 \alpha$. 
Then, for any $E \in |L|_{\mathbb{R}}$, the pair $(X, \Delta+2\alpha E)$ is klt. 
By Corollary \ref{cor--mmp-boundary-polytope} and replacing $\alpha$ if necessary, we may assume that any $(K_{X}+\Delta)$-MMP is also a $(K_{X}+\Delta+2\alpha E)$-MMP for every $E \in |L|_{\mathbb{R}}$. 

From now on, we show that this $\alpha$ satisfies the condition of Lemma \ref{lem--base-klt-threshold}. 
Pick any element 

$$\Bigl((X,\Delta) \overset{\sigma'}{\dashrightarrow} (X',\Delta') \overset{\pi'}{\longrightarrow}Z', R'\Bigr)$$ 
of $\mathfrak{M}_{X}$ with the induced small birational map $\psi \colon Z_{0} \dashrightarrow Z'$ such that $\pi' \circ \phi=\psi \circ \pi_{0}$ as in Theorem \ref{thm--two-mfs-small}. 
Now we have the following diagram
$$
\xymatrix{
&W\ar[dl]_{f}\ar[dr]^{f_{0}}\\
X\ar@{-->}[rr]^-{\sigma_{0}}&&X_{0}\ar@{-->}[rr]^-{\phi}\ar[d]_{\pi_{0}}&&X'\ar[d]^{\pi'}\\
&&Z_{0}\ar@{-->}[rr]_-{\psi}&&Z'.
}
$$
Pick any element $G_{Z_{0}} \in |L_{Z_{0}}|_{\mathbb{R}}$. 
We put 
$$G_{Z'} :=\psi_{*}G_{Z_{0}}, \qquad G' :=\pi'^{*}G_{Z'}, \qquad {\rm and} \qquad G:=f_{*}f_{0}^{*}\phi^{-1}_{*}G'.$$ 
Since $\phi$ and $\psi$ are small and the images of all vertical prime divisors with respect to $\pi_{0}$ and $\pi'$ are prime divisors (see Lemma \ref{lem--MFS-basic}), we have $\phi^{-1}_{*}G'=\pi_{0}^{*}G_{Z_{0}}$. 
Then $G \in |L|_{\mathbb{R}}$, and $(X, \Delta+2\alpha G)$ is klt by the above argument. 
Moreover, $(X',\Delta'+2\alpha G')$ is klt since any $(K_{X}+\Delta)$-MMP is also a $(K_{X}+\Delta+2\alpha G)$-MMP, which was discussed above. 

We pick a sufficiently general $\pi'$-ample $\mathbb{R}$-divisor $H' \sim_{\mathbb{R},\,Z'}-(K_{X'}+\Delta'+2\alpha G')$, and we consider the klt-trivial fibration $\pi' \colon (X',\Delta'+2\alpha G'+H') \to Z'$. 
By the canonical bundle formula for $\mathbb{R}$-boundary divisors \cite[Theorem 1.7]{fujino-hashizume-adjunction}, we get a generalized klt pair $(Z',B'+M')$ on $Z'$ with the boundary part $B'$ and the nef part $M'$ (cf.~\cite{bz}). 
Since $Z'$ is $\mathbb{Q}$-factorial, $(Z',B')$ is klt. 
Moreover we have $B' \geq 2\alpha G_{Z'}$. 
Indeed, for any prime divisor $P'$ on $Z'$, the lc threshold of $\pi'^{*}P'$ with respect to $(X',\Delta'+2\alpha G'+H')$ over the generic point of $P'$, which we denote by $t_{P'}$, is not greater than $1- {\rm coeff}_{P'}(2\alpha G_{Z'})$. 
Then 
$${\rm coeff}_{P'}(2\alpha G_{Z'}) \leq 1-t_{P'}={\rm coeff}_{P'}(B').$$
Thus,  $B' \geq 2\alpha G_{Z'}$, and therefore $(Z',  2\alpha G_{Z'})$ is klt. 
This completes the proof. 
\end{proof}

\begin{cor}\label{cor--form-small-family}
With notation as in Setup \ref{setup--base-mfs}, we consider the set
$$\mathfrak{F}:=\left\{\psi \colon Z_{0} \dashrightarrow Z' \,\middle|\, \Bigl((X,\Delta) \overset{\sigma'}{\dashrightarrow} (X',\Delta') \overset{\pi'}{\longrightarrow}Z', R'\Bigr) \in \mathfrak{M}_{X},\, \pi' \circ \phi=\psi \circ \pi_{0}\right\}.$$
Then the following properties hold.
\begin{enumerate}[(i)]
\item\label{cor--form-small-family-(i)}
every $\psi \colon Z_{0} \dashrightarrow Z' \in \mathfrak{F}$ is a small birational map and $Z'$ is a projective $\mathbb{Q}$-factorial variety, and 
\item\label{cor--form-small-family-(ii)}
for any effective $\mathbb{R}$-divisor $L_{Z_{0}}$ on $Z_{0}$, there exists $\alpha \in \mathbb{Q}_{>0}$ such that for every element $\psi \colon Z_{0} \dashrightarrow Z' \in \mathfrak{F}$, the pair $(Z',2\alpha \psi_{*}G_{Z_{0}})$ is klt for all $G_{Z_{0}} \in |L_{Z_{0}}|_{\mathbb{R}}$. 
\end{enumerate}
\end{cor}

\begin{proof}
This corollary immediately follows from Theorem \ref{thm--two-mfs-small} and Lemma \ref{lem--base-klt-threshold}. 
\end{proof}

\section{Inductive procedure}\label{sec--ind}

In this section we prove the key statement for the main results. 

Motivated by Setup \ref{setup--base-mfs} and Corollary \ref{cor--form-small-family}, we introduce the following abstract setting.

\begin{defn}\label{defn--ind-small-family}
Let $Y_{0}$ be a projective $\mathbb{Q}$-factorial variety. 
Let $\mathfrak{F}_{Y_0}$ be a set consisting of small birational maps
$$\psi \colon Y_0\dashrightarrow Y$$
to projective $\mathbb{Q}$-factorial varieties $Y$. 
We assume ${\rm id}_{Y_{0}} \in \mathfrak{F}_{Y_{0}}$ and the following condition. 
\begin{equation*}\tag{$*$}\label{defn--ind-small-family-(*)}
\begin{split}
&\text{For any effective $\mathbb{R}$-divisor $L_{Y_{0}}$ on $Y_{0}$, there exists $\alpha \in \mathbb{Q}_{>0}$ such that the pair}\\
&\text{$(Y,2\alpha \psi_{*}G_{Y_{0}})$ is klt for every $G_{Y_{0}} \in |L_{Y_{0}}|_{\mathbb{R}}$ and $\psi \colon Y_{0} \dashrightarrow Y \in \mathfrak{F}_{Y_{0}}$.} 
\end{split}
\end{equation*}
We note that (\ref{defn--ind-small-family-(*)}) is preserved after replacing $Y_{0}$ by any member of the small birational family. 

In the rest of this paper, we fix this $\alpha$ for each $L_{Y_{0}}$.    
\end{defn}

\begin{lem}\label{lem--base-bound-Cartier}
With notation as in Definition \ref{defn--ind-small-family}, let $L_{Y_{0}}$ be an effective big $\mathbb{Q}$-divisor on $Y_{0}$, and let $\alpha$ be as in {\rm(}\ref{defn--ind-small-family-(*)}{\rm)}.  
Then there exists $I \in \mathbb{Z}_{>0}$, depending only on $\mathfrak{F}_{Y_{0}}$, $L_{Y_{0}}$, and $\alpha$,  satisfying the following. 
Let $\psi \colon Y_{0} \dashrightarrow Y \in \mathfrak{F}_{Y_{0}}$ be an element. 
Set $L_{Y}:=\psi_{*}L_{Y_{0}}$. 
Then, for any sequence of steps of a $(K_{Y}+\alpha L_{Y})$-MMP 
$(Y, \alpha L_{Y}) \dashrightarrow (Y', \alpha L_{Y'})$ 
and any Weil divisor $D_{Y'}$ on $Y'$, the multiple $ID_{Y'}$ is Cartier. 
\end{lem}

\begin{proof}
We use log canonical volumes (\cite[Definition 3.7 and Definition 3.9]{han-qi-zhuang}). 
We consider the graded linear system 
\begin{equation*}V_{m}:=\left\{\begin{array}{ccl}H^{0}(Y,\mathcal{O}_{Y}(m\alpha L_{Y}))&\qquad&(\text{$m \alpha L_{Y}$ is a Cartier divisor})\\0&\qquad& (\text{otherwise}).\end{array}\right.\end{equation*}
Since $L_{Y}$ is big and the property of $\alpha$ implies that $(Y,2\alpha G_{Y})$ is klt for any $G_{Y} \in |L_{Y}|_{\mathbb{Q}}$, it follows that $V_{\bullet}$ is admissible (\cite[Definition 3.7]{han-qi-zhuang}). 
Then we have 
$$\widehat{\rm Vol}_{(Y,0)}(\alpha L_{Y})\geq \alpha^{{\rm dim}\,Y} \cdot{\rm vol}(L_{Y}).$$ 
Since $\psi \colon Y_{0} \dashrightarrow Y$ is small, we have ${\rm vol}(L_{Y})={\rm vol}(L_{Y_{0}})$. 
Therefore 
$$\widehat{\rm Vol}_{(Y,0)}(\alpha L_{Y})\geq \alpha^{{\rm dim}\,Y_{0}} \cdot{\rm vol}(L_{Y_{0}}).$$
Then Lemma \ref{lem--base-bound-Cartier} follows from \cite[Lemma 3.11 and Lemma 3.13]{han-qi-zhuang}.  
\end{proof}

\begin{thm}\label{thm--base-finite-intnum-func}
With notation as in Definition \ref{defn--ind-small-family}, let $L_{Y_{0}}:=\sum_{i=1}^{n}a_{i}L_{i}$ be an effective big $\mathbb{Q}$-divisor on $Y_{0}$ such that
\begin{itemize}
\item
each $L_{i}$ is an effective big $\mathbb{Q}$-divisor on $Y_{0}$ and $a_{i} \in \mathbb{Q}_{>0}$, and
\item
the rational polytope in $N^{1}(Y_{0})_{\mathbb{R}}$ spanned by $L_{1},\,\cdots,\,L_{n}$ has dimension $\rho(Y_{0})$. 
\end{itemize}
For every $\psi \colon Y_{0} \dashrightarrow Y \in \mathfrak{F}_{Y_{0}}$, we put $\Delta_{Y}:=\alpha \psi_{*}L_{Y_{0}}$, where $\alpha$ is as in {\rm(}\ref{defn--ind-small-family-(*)}{\rm)}. 
Then the set of functions 
\begin{equation*}
\mathcal{S}:=\left\{\zeta_{Y',R'}\circ \psi_{*}\,\,\middle|
\begin{array}{rl}
\bullet\!\!&\text{$\psi \colon Y_{0} \dashrightarrow Y \in \mathfrak{F}_{Y_{0}}$,}\\
\bullet\!\!&\text{$(Y,\Delta_{Y}) \dashrightarrow (Y', \Delta_{Y'})$ is a $(K_{Y}+ \Delta_{Y})$-MMP,}\\
\bullet\!\!&\text{$R'$ is a $(K_{Y'}+\Delta_{Y'})$-negative extremal ray}\end{array}
\right\},
\end{equation*}
where $\zeta_{Y',R'}$ are the functions defined in Definition \ref{defn--intnum-func}, is a finite set. 
\end{thm}

\begin{proof}
We apply the arguments in the proofs of Corollary \ref{cor--mmp-boundary-polytope} and Theorem \ref{thm--finite-intnum-func} uniformly to all the elements of $\mathfrak{F}_{Y_{0}}$. 

By the property of $\alpha$, the pair $(Y,2\Delta_{Y})$ is klt for every element $\psi \colon Y_{0}\dashrightarrow Y$ of $\mathfrak{F}_{Y_{0}}$.
Hence, $(Y,\psi_{*}\tilde{\Delta}_{Y_{0}})$ is klt for every $\tilde{\Delta}_{Y_{0}}$ such that $0\leq \tilde{\Delta}_{Y_{0}}\leq 2\alpha L_{Y_{0}}$.
Furthermore, the positive integer $I$ in Lemma \ref{lem--base-bound-Cartier} is independent of the choice of an element of  $\mathfrak{F}_{Y_{0}}$. 
Thus, we may carry out the proof of Corollary \ref{cor--mmp-boundary-polytope} uniformly for all $\psi\colon Y_{0}\dashrightarrow Y\in\mathfrak{F}_{Y_{0}}$. We use the fixed decomposition $L_{Y_{0}}=\sum_{i=1}^{n}a_iL_i$, the set
$$\{\tilde{\Delta}_{Y_{0}}\mid 0\leq\tilde{\Delta}_{Y_{0}}\leq2\alpha L_{Y_{0}}\},$$
and its image under $\psi_{*}$.
Then we can find finitely many effective $\mathbb{Q}$-divisors $\Delta_{1},\,\cdots ,\,\Delta_{p}$ on $Y_{0}$ satisfying the following. 
Let $\mathcal{C}_{Y_{0}} \subset N^{1}(Y_{0})_{\mathbb{R}}$ be the rational polytope spanned by $\Delta_{1},\,\cdots ,\,\Delta_{p}$, and for every element $\psi \colon Y_{0}\dashrightarrow Y \in \mathfrak{F}_{Y_{0}}$, we put $B_{j}:=\psi_{*}\Delta_{j}$ $(1 \leq j \leq p)$ and $\mathcal{C}:=\psi_{*}\mathcal{C}_{Y_{0}}$. 
Then
\begin{itemize}
\item
$(Y,B_{j})$ is klt for every $1 \leq j \leq p$, 
\item
${\rm dim}\, \mathcal{C} =\rho(Y)$, 
\item
for any klt pair $(Y,B)$ such that $B \in \mathcal{C}$ as a numerical class, any sequence of steps of a $(K_{Y}+\Delta_{Y})$-MMP 
$$\phi \colon(Y,\Delta_{Y}) \dashrightarrow (Y',\Delta_{Y'}),$$
and any $(K_{Y'}+\Delta_{Y'})$-negative extremal ray $R'$ of $\overline{\rm NE}(Y')$, we have 
$$ (K_{Y'}+\phi_{*}B)\cdot R'<0,$$ 
in particular, any $(K_{Y}+\Delta_{Y})$-MMP is also a $(K_{Y}+B)$-MMP, and 
\item
$\Delta_{Y} \in {\rm int}(\mathcal{C})$ as a numerical class. 
\end{itemize}
Since $\Delta_{1},\,\cdots,\,\Delta_{p}$ are fixed $\mathbb{Q}$-divisors on $Y_{0}$, there exists $\ell' \in \mathbb{Z}_{>0}$ independent of $Y$ such that $\ell' B_{j}$ is a Weil divisor on $Y$ for every $1\leq j\leq p$. 
Using this $\ell'$ and the positive integer $I$ in Lemma \ref{lem--base-bound-Cartier}, we may apply the proof of Theorem \ref{thm--finite-intnum-func} uniformly to all $Y$, and Theorem \ref{thm--base-finite-intnum-func} follows.  
\end{proof}

\begin{thm}\label{thm--base-finite-cont-div}
With notation as in Definition \ref{defn--ind-small-family}, let $L_{Y_{0}}:=\sum_{i=1}^{n}a_{i}L_{i}$ be an effective big $\mathbb{Q}$-divisor on $Y_{0}$ such that
\begin{itemize}
\item
each $L_{i}$ is an effective big $\mathbb{Q}$-divisor on $Y_{0}$ and $a_{i} \in \mathbb{Q}_{>0}$, and
\item
the rational polytope in $N^{1}(Y_{0})_{\mathbb{R}}$ spanned by $L_{1},\,\cdots,\,L_{n}$ has dimension $\rho(Y_{0})$. 
\end{itemize}
For every $\psi \colon Y_{0} \dashrightarrow Y \in \mathfrak{F}_{Y_{0}}$, put $\Delta_{Y}:=\alpha \psi_{*}L_{Y_{0}}$, where $\alpha$ is as in {\rm(}\ref{defn--ind-small-family-(*)}{\rm)}. 
Then there exists a finite set $\mathcal{Q}_{0}$ of prime divisors on $Y_{0}$ satisfying the following. 
Let $\psi \colon Y_{0} \dashrightarrow Y \in \mathfrak{F}_{Y_{0}}$ be an element and let $(Y,\Delta_{Y}) \dashrightarrow (Y', \Delta_{Y'})$ be a sequence of steps of a $(K_{Y}+ \Delta_{Y})$-MMP. 
Let $P_{Y_{0}}$ be a prime divisor on $Y_{0}$ contracted by $Y_{0} \dashrightarrow Y \dashrightarrow Y'$.
Then $P_{Y_{0}} \in \mathcal{Q}_{0}$. 
\end{thm}

\begin{proof}
The argument in the proof of Theorem \ref{thm--finite-div-cont-mmp} works in our situation because we may use Theorem \ref{thm--base-finite-intnum-func}. 
\end{proof}

\begin{thm}\label{thm--base-mfs-small}
With notation as in Definition \ref{defn--ind-small-family}, let $L_{Y_{0}}:=\sum_{i=1}^{n}a_{i}L_{i}$ be an effective big $\mathbb{Q}$-divisor on $Y_{0}$ such that
\begin{itemize}
\item
each $L_{i}$ is an effective big $\mathbb{Q}$-divisor on $Y_{0}$ and $a_{i} \in \mathbb{Q}_{>0}$, and
\item
the rational polytope in $N^{1}(Y_{0})_{\mathbb{R}}$ spanned by $L_{1},\,\cdots,\,L_{n}$ has dimension $\rho(Y_{0})$. 
\end{itemize}
Fix two elements $\psi' \colon Y_{0} \dashrightarrow Y'$ and $\psi'' \colon Y_{0} \dashrightarrow Y''$ of $\mathfrak{F}_{Y_{0}}$, and put $\Delta_{Y'}:=\alpha \psi'_{*}L_{Y_{0}}$ and $\Delta_{Y''}:=\alpha \psi''_{*}L_{Y_{0}}$, where $\alpha$ is as in {\rm(}\ref{defn--ind-small-family-(*)}{\rm)}. 
Let 
$$(Y',\Delta_{Y'}) \dashrightarrow (\tilde{Y}',\Delta_{\tilde{Y}'}) \overset{h'}{\longrightarrow} T' \qquad {\rm and} \qquad (Y'',\Delta_{Y''}) \dashrightarrow (\tilde{Y}'',\Delta_{\tilde{Y}''}) \overset{h''}{\longrightarrow} T''$$ 
be sequences of steps of a $(K_{Y'}+\Delta_{Y'})$-MMP and a $(K_{Y''}+\Delta_{Y''})$-MMP terminating with Mori fiber spaces $h' \colon (\tilde{Y}',\Delta_{\tilde{Y}'}) \to T'$ and $h'' \colon (\tilde{Y}'',\Delta_{\tilde{Y}''}) \to T''$, respectively. 
Let $\tilde{R}'$ and $\tilde{R}''$ be extremal rays that define $h'$ and $h''$, respectively. 
Suppose that 
\begin{itemize}
\item
the induced birational map $\phi \colon (\tilde{Y}',\Delta_{\tilde{Y}'}) \dashrightarrow (\tilde{Y}'',\Delta_{\tilde{Y}''})$ is small, and
\item
$\zeta_{\tilde{Y}',\tilde{R}'} \circ \psi'_{*}= c \zeta_{\tilde{Y}'',\tilde{R}''} \circ \psi''_{*}$ as functions (see Definition \ref{defn--intnum-func}) for some $c \in \mathbb{R}_{>0}$. 
\end{itemize}
Then there exists a small birational map $\theta \colon T' \dashrightarrow T''$ such that $h'' \circ \phi=\theta \circ h'$. 
\end{thm}

\begin{proof}
The argument in the proof of Theorem \ref{thm--two-mfs-small} works after replacing the common starting model there by the common marked model $Y_{0}$. 
\end{proof}

\begin{thm}\label{thm--ind-klt-threshold}
With notation as in Definition \ref{defn--ind-small-family}, let $L_{Y_{0}}:=\sum_{i=1}^{n}a_{i}L_{i}$ be an effective big $\mathbb{Q}$-divisor on $Y_{0}$ such that
\begin{itemize}
\item
each $L_{i}$ is an effective big $\mathbb{Q}$-divisor on $Y_{0}$ and $a_{i} \in \mathbb{Q}_{>0}$, and
\item
the rational polytope in $N^{1}(Y_{0})_{\mathbb{R}}$ spanned by $L_{1},\,\cdots,\,L_{n}$ has dimension $\rho(Y_{0})$. 
\end{itemize}
Put $\Delta_{Y_{0}}:=\alpha L_{Y_{0}}$. 
Suppose that $K_{Y_{0}}+\Delta_{Y_{0}}$ is not pseudo-effective.
Let 
$$(Y_{0},\Delta_{Y_{0}}) \overset{\sigma_{0}}{\dashrightarrow} (Y'_{0},\Delta'_{0}) \overset{h_{0}}{\longrightarrow}T_{0}$$ 
be a sequence of steps of a $(K_{Y_{0}}+\Delta_{Y_{0}})$-MMP terminating with a Mori fiber space $h_{0} \colon (Y'_{0},\Delta'_{0}) \to T_{0}$, and let $R_{0}\subset \overline{\rm NE}(Y'_{0})$ be a $(K_{Y'_{0}}+\Delta'_{0})$-negative extremal ray defining $h_{0}$. 
Let $H_{T_{0}}$ be an effective $\mathbb{R}$-divisor on $T_{0}$. 
Then there exists $\beta \in \mathbb{Q}_{>0}$ satisfying the following. 
We define $\mathfrak{N}$ to be the set of pairs
$$\Bigl((Y,\Delta_{Y}) \overset{\sigma'}{\dashrightarrow} (Y',\Delta_{Y'}) \overset{h'}{\longrightarrow}T', R'\Bigr),$$
where $(Y,\Delta_{Y}) \overset{\sigma'}{\dashrightarrow} (Y',\Delta_{Y'}) \overset{h'}{\longrightarrow}T'$ is a sequence of steps of a $(K_{Y}+\Delta_{Y})$-MMP that terminates with a Mori fiber space $h' \colon (Y',\Delta_{Y'}) \to T'$ for some $\psi \colon Y_{0}\dashrightarrow Y \in \mathfrak{F}_{Y_{0}}$ with $\Delta_{Y}:=\psi_{*}\Delta_{Y_{0}}$, and $R'$ is the $(K_{Y'}+\Delta_{Y'})$-negative extremal ray defining $h'$, such that
\begin{itemize}
\item
the induced birational map $\phi \colon Y'_{0} \dashrightarrow Y'$ is small, and 
\item
$\zeta_{Y'_{0},R_{0}}=c\zeta_{Y',R'} \circ \psi_{*}$ for some $c \in \mathbb{R}_{>0}$. 
\end{itemize} 
We note that there exists a small birational map $\theta \colon T_{0} \dashrightarrow T'$ such that $h' \circ \phi=\theta \circ h_{0}$. 
Then, for any $G_{T_{0}} \in |H_{T_{0}}|_{\mathbb{R}}$, the pair $(T',2\beta \theta_{*}G_{T_{0}})$ is klt. 
\end{thm}

\begin{proof}
The proof of Lemma \ref{lem--base-klt-threshold} is valid for our situation. 
For the reader's convenience, we write the detailed proof. 

Let $f \colon W \to Y_{0}$ and $f' \colon W \to Y'_{0}$ be a common resolution of $\sigma_{0} \colon Y_{0} \dashrightarrow Y'_{0}$. 
We put $H_{0}:=f_{*}f'^{*}h_{0}^{*}H_{T_{0}}$. 
Since $Y_{0}$ is $\mathbb{Q}$-factorial, $H_{0}$ is an effective $\mathbb{R}$-Cartier divisor on $Y_{0}$. 
By the property {\rm(}\ref{defn--ind-small-family-(*)}{\rm)} in Definition \ref{defn--ind-small-family}, there exists $\beta \in \mathbb{Q}_{>0}$ such that $(Y,4\beta \psi_{*}E_{0})$ is klt for every $E_{0} \in |H_{0}|_{\mathbb{R}}$ and every $\psi \colon Y_{0} \dashrightarrow Y \in \mathfrak{F}_{Y_{0}}$. 
Since $(Y, 2\Delta_{Y})$ is klt for every $\psi \colon Y_{0} \dashrightarrow Y \in \mathfrak{F}_{Y_{0}}$, we see that $(Y, \Delta_{Y}+2\beta \psi_{*}E_{0})$ is klt for every $E_{0} \in |H_{0}|_{\mathbb{R}}$ and every $\psi \colon Y_{0} \dashrightarrow Y \in \mathfrak{F}_{Y_{0}}$. 
By the same argument as in the proof of Theorem \ref{thm--base-finite-intnum-func} (see also Corollary \ref{cor--mmp-boundary-polytope}) and replacing $\beta$ if necessary, we may assume that any $(K_{Y}+\Delta_{Y})$-MMP is also a $(K_{Y}+\Delta_{Y}+2\beta E_{Y})$-MMP for every $E_{Y} \in |\psi_{*}H_{0}|_{\mathbb{R}}$. 

Now we have the following diagram
$$
\xymatrix{
&W\ar[dl]_{f}\ar[dr]^{f'}\\
Y_{0}\ar@{-->}[rr]^-{\sigma_{0}}&&Y'_{0}\ar@{-->}[rr]^-{\phi}\ar[d]_{h_{0}}&&Y'\ar[d]^{h'}\\
&&T_{0}\ar@{-->}[rr]_-{\theta}&&T'.
}
$$
Pick any element $G_{T_{0}} \in |H_{T_{0}}|_{\mathbb{R}}$. 
We put 
$$G_{T'} :=\theta_{*}G_{T_{0}}, \qquad G' :=h'^{*}G_{T'}, \qquad {\rm and} \qquad G_{0}:=f_{*}f'^{*}\phi^{-1}_{*}G'.$$ 
Since $\phi$ and $\theta$ are small and the images of all vertical prime divisors with respect to $h_{0}$ and $h'$ are prime divisors (Lemma \ref{lem--MFS-basic}), we have $\phi^{-1}_{*}G'=h_{0}^{*}G_{T_{0}}$. 
Then $G_{0} \in |H_{0}|_{\mathbb{R}}$, and $(Y, \Delta_{Y}+2\beta \psi_{*}G_{0})$ is klt by the above argument. 
Moreover, $(Y', \Delta_{Y'}+2\beta G')$ is klt since any $(K_{Y}+\Delta_{Y})$-MMP is also a $(K_{Y}+\Delta_{Y}+2\beta \psi_{*}G_{0})$-MMP, which was discussed above. 

We pick a sufficiently general $h'$-ample $\mathbb{R}$-divisor $H' \sim_{\mathbb{R},\,T'}-(K_{Y'}+\Delta_{Y'}+2\beta G')$, and we consider the klt-trivial fibration $h' \colon (Y', \Delta_{Y'}+2\beta G'+H') \to T'$. 
By the canonical bundle formula for $\mathbb{R}$-boundary divisors \cite[Theorem 1.7]{fujino-hashizume-adjunction}, we get a generalized klt pair $(T',B'+M')$ on $T'$ with the boundary part $B'$ and the nef part $M'$ (cf.~\cite{bz}).  
Since $T'$ is $\mathbb{Q}$-factorial, $(T',B')$ is klt. 
Moreover we have $B' \geq 2\beta G_{T'}$. 
Indeed, for any prime divisor $P'$ on $T'$, the lc threshold of $h'^{*}P'$ with respect to $(Y', \Delta_{Y'}+2\beta G'+H')$ over the generic point of $P'$, which we denote by $t_{P'}$, is not greater than $1- {\rm coeff}_{P'}(2\beta G_{T'})$. 
Then 
$${\rm coeff}_{P'}(2\beta G_{T'}) \leq 1-t_{P'}={\rm coeff}_{P'}(B').$$
Thus,  $B' \geq 2\beta G_{T'}$, and therefore $(T',  2\beta G_{T'})$ is klt. 
This completes the proof. 
\end{proof}

\begin{cor}\label{cor--ind-form-small-family}
We use the notation in Definition \ref{defn--ind-small-family} and we retain the setup and notation of Theorem \ref{thm--ind-klt-threshold}. 
As in Theorem \ref{thm--ind-klt-threshold}, we define $\mathfrak{N}$ to be the set of pairs
$$\Bigl((Y,\Delta_{Y}) \overset{\sigma'}{\dashrightarrow} (Y',\Delta_{Y'}) \overset{h'}{\longrightarrow}T', R'\Bigr),$$
where $(Y,\Delta_{Y}) \overset{\sigma'}{\dashrightarrow} (Y',\Delta_{Y'}) \overset{h'}{\longrightarrow}T'$ is a sequence of steps of a $(K_{Y}+\Delta_{Y})$-MMP that terminates with a Mori fiber space $h' \colon (Y',\Delta_{Y'}) \to T'$ for some $\psi \colon Y_{0}\dashrightarrow Y \in \mathfrak{F}_{Y_{0}}$, and $R'$ is the $(K_{Y'}+\Delta_{Y'})$-negative extremal ray defining $h'$, such that
\begin{itemize}
\item
the induced birational map $\phi \colon Y'_{0} \dashrightarrow Y'$ is small, and 
\item
$\zeta_{Y'_{0},R_{0}}=c\zeta_{Y',R'} \circ \psi_{*}$ for some $c \in \mathbb{R}_{>0}$. 
\end{itemize} 
We note that there exists a small birational map $\theta \colon T_{0} \dashrightarrow T'$ such that $h' \circ \phi=\theta \circ h_{0}$. 
We consider the set
$$\mathfrak{G}:=\left\{\theta \colon T_{0} \dashrightarrow T' \,\middle|\, \Bigl((Y,\Delta_{Y}) \overset{\sigma'}{\dashrightarrow} (Y',\Delta_{Y'}) \overset{h'}{\longrightarrow}T', R'\Bigr) \in \mathfrak{N},\, h' \circ \phi=\theta \circ h_{0}\right\}.$$
Then the following properties hold.
\begin{enumerate}[(i)]
\item\label{cor--ind-form-small-family-(i)}
every $\theta \colon T_{0} \dashrightarrow T' \in \mathfrak{G}$ is a small birational map and $T'$ is a projective $\mathbb{Q}$-factorial variety, and 
\item\label{cor--ind-form-small-family-(ii)}
for any effective $\mathbb{R}$-divisor $H_{T_{0}}$ on $T_{0}$, there exists $\beta \in \mathbb{Q}_{>0}$ such that for every $\theta \colon T_{0} \dashrightarrow T' \in \mathfrak{G}$, the pair $(T',2\beta \theta_{*}G_{T_{0}})$ is klt for all $G_{T_{0}} \in |H_{T_{0}}|_{\mathbb{R}}$. 
\end{enumerate}
\end{cor}

\begin{proof}
This corollary immediately follows from Theorem \ref{thm--base-mfs-small} and Theorem \ref{thm--ind-klt-threshold}. 
\end{proof}

\section{Inverse of minimal model program}\label{sec--inv-mmp}

In this section, we study inverse MMP steps.

\begin{lem}\label{lem--inv-divcont}
Let $(X,\Delta)$ be a projective $\mathbb{Q}$-factorial lc pair.  
Fix an $\mathbb{R}$-linear function $\zeta \colon N^{1}(X)_{\mathbb{R}} \to \mathbb{R}$. 
Fix another projective $\mathbb{Q}$-factorial lc pair $(\tilde{X},\tilde{\Delta})$ with a birational contraction $\phi\colon X\dashrightarrow \tilde{X}$ such that $\tilde{\Delta}=\phi_{*}\Delta$. 
Then the set
$$\left\{\psi' \colon (X',\Delta') \to (\tilde{X},\tilde{\Delta})\,\middle|\begin{array}{l}\text{$(X',\Delta')$ is a projective $\mathbb{Q}$-factorial lc pair,}\\
\text{$\psi'^{-1}\circ \phi\colon (X,\Delta)\dashrightarrow (X',\Delta')$ is a birational contraction,}\\
\text{$\psi'$ is a divisorial contraction defined by a $(K_{X'}+\Delta')$-}\\
\text{negative extremal ray $R'$ such that $\zeta_{X',R'}\circ (\psi'^{-1} \circ \phi)_{*}=\zeta$\!}\end{array}\right\}$$
consists of at most one element up to isomorphism over $\tilde{X}$ compatible with the markings. 
\end{lem}

\begin{rem}
In Lemma \ref{lem--inv-divcont}, we do not assume that $\psi'^{-1}\circ \phi \colon (X,\Delta)\dashrightarrow (X',\Delta')$ is a sequence of steps of a $(K_{X}+\Delta)$-MMP. 
\end{rem}

\begin{proof}
For any $\psi' \colon (X',\Delta') \to (\tilde{X},\tilde{\Delta})$ and $\psi'' \colon (X'',\Delta'') \to (\tilde{X},\tilde{\Delta})$ as in the set, the equality
$$\zeta_{X',R'}\circ (\psi'^{-1} \circ \phi)_{*}=\zeta= \zeta_{X'',R''}\circ (\psi''^{-1} \circ \phi)_{*}$$
implies that the induced map $(X',\Delta') \dashrightarrow (X'',\Delta'')$ is small. 
Indeed, pick a prime divisor $P$ on $X$, and let $P'$ (resp.~$P''$) be the birational transform of $P$ on $X'$ (resp.~$X''$). 
By Definition \ref{defn--intnum-func}, we have
\begin{equation*}
\zeta_{X',R'}(P')\,\left\{\begin{array}{ll}
<0& \text{$P'$ is $\psi'$-exceptional}\\
\geq0 & (\text{otherwise})
\end{array}\right.
\; {\rm and} \;
\zeta_{X'',R''}(P'')\,\left\{\begin{array}{ll}
<0& \text{$P''$ is $\psi''$-exceptional}\\
\geq0 & (\text{otherwise}).
\end{array}\right.
\end{equation*}
Therefore, $P'$ is $\psi'$-exceptional if and only if $P''$ is $\psi''$-exceptional for any prime divisor $P$ on $X$. 
This shows that $(X',\Delta') \dashrightarrow (X'',\Delta'')$ is small. 
Since both $-(K_{X'}+\Delta')$ and $-(K_{X''}+\Delta'')$ are ample over $\tilde{X}$, the negativity lemma shows that $(X',\Delta') \dashrightarrow (X'',\Delta'')$ is an isomorphism. 
\end{proof}

\begin{lem}\label{lem--inv-flip}
Let $(X,\Delta)$ be a projective $\mathbb{Q}$-factorial lc pair. 
Fix an $\mathbb{R}$-linear function $\zeta \colon N^{1}(X)_{\mathbb{R}} \to \mathbb{R}$. 
Fix another projective $\mathbb{Q}$-factorial lc pair $(\tilde{X},\tilde{\Delta})$ with a birational contraction $\phi\colon X\dashrightarrow \tilde{X}$ such that $\tilde{\Delta}=\phi_{*}\Delta$. 
Then 
$$\left\{\psi' \colon (X',\Delta') \dashrightarrow (\tilde{X},\tilde{\Delta})\,\middle|\,\begin{array}{l}\text{$(X',\Delta')$ is a projective $\mathbb{Q}$-factorial lc pair,}\\
\text{$\psi'$ is a $(K_{X'}+\Delta')$-flip defined by}\\\text{an extremal ray $R'$ such that $\zeta_{X',R'}\circ (\psi'^{-1} \circ \phi)_{*}=\zeta$}\end{array}\right\}$$
consists of at most one element up to isomorphism compatible with the markings. 
\end{lem}

\begin{rem}
In Lemma \ref{lem--inv-flip}, we do not assume that $\psi'^{-1}\circ \phi \colon (X,\Delta)\dashrightarrow (X',\Delta')$ is a sequence of steps of a $(K_{X}+\Delta)$-MMP. 
\end{rem}

\begin{proof}
Pick two elements
$$\psi' \colon (X',\Delta') \dashrightarrow (\tilde{X},\tilde{\Delta}) \qquad {\rm and} \qquad \psi'' \colon (X'',\Delta'') \dashrightarrow (\tilde{X},\tilde{\Delta})$$
with extremal rays $R'$ and $R''$ in the set of Lemma \ref{lem--inv-flip}. 
Let $X' \to V'$ (resp.~$X'' \to V''$) be the flipping contraction, and let $\tilde{X} \to V'$ (resp.~$\tilde{X} \to V''$) be the flip. 

We prove the existence of an isomorphism $V' \to V''$ compatible with $\tilde{X} \to V'$ and $\tilde{X} \to V''$.  
Let $H_{V'}$ be an ample $\mathbb{R}$-divisor on $V'$, and let $\tilde{H}_{V'}$ be the pullback to $\tilde{X}$. 
We put $H' := \psi'^{-1}_{*}\tilde{H}_{V'}$ and $H'' := \psi''^{-1}_{*}\tilde{H}_{V'}$. 
Then $(H'\cdot R') =0$ since we have $H' \sim_{\mathbb{R},\,V'}0$. 
Since 
$$\zeta_{X'',R''}\circ (\psi''^{-1}\circ \phi)_{*}=\zeta_{X',R'}\circ (\psi'^{-1}\circ \phi)_{*},$$ 
we have $(H''\cdot R'') =0$. 
This shows that $H'' \sim_{\mathbb{R},\,V''}0$, and therefore  $\tilde{H}_{V'} \sim_{\mathbb{R},\,V''}0$. 
Then, for any curve $\tilde{C} \subset \tilde{X}$ contracted by $\tilde{X} \to V''$, if we put $C'$ as the pushforward of $\tilde{C}$ on $V'$, we have
$$(H_{V'}\cdot C') =(\tilde{H}_{V'} \cdot \tilde{C})=0.$$
Since $H_{V'}$ is ample, we see that $\tilde{C}$ is contracted by $\tilde{X} \to V'$. 
Thus, any curve $\tilde{C} \subset \tilde{X}$ contracted by $\tilde{X} \to V''$ is also contracted by $\tilde{X} \to V'$. 
From this fact, we obtain a morphism $V'' \to V'$ compatible with $\tilde{X} \to V''$ and $\tilde{X} \to V'$. 
By the same argument, we obtain a morphism $V' \to V''$ compatible with $\tilde{X} \to V'$ and $\tilde{X} \to V''$. 
This shows that the morphism $V' \to V''$ is an isomorphism. 
Since $-(K_{X'}+\Delta')$ and $-(K_{X''}+\Delta'')$ are relatively ample and the birational map $\psi''^{-1}\circ \psi' \colon (X',\Delta') \dashrightarrow (X'',\Delta'')$ is small, the negativity lemma implies that $\psi''^{-1}\circ \psi'$ is an isomorphism compatible with the markings. 
\end{proof}

\begin{lem}\label{lem--inv-mfs}
Let $(X,\Delta) \to Z$ and $(X',\Delta') \to Z$ be contractions from projective lc pairs. 
Let $\phi \colon (X,\Delta)\dashrightarrow (X',\Delta')$ be a small birational map over $Z$ such that $\phi_{*}\Delta=\Delta'$. 
If $-(K_{X}+\Delta)$ and $-(K_{X'}+\Delta')$ are ample over $Z$, then $\phi$ is an isomorphism over $Z$.
\end{lem}

\begin{proof}
This easily follows from the negativity lemma. 
\end{proof}

\begin{prop}\label{prop--num-step-mmp}
Let $(X,\Delta)$ be a projective $\mathbb{Q}$-factorial lc pair such that the set
\begin{equation*}
\mathcal{S}:=\left\{\zeta_{X',R'}\,\middle|
\begin{array}{l}\text{$(X,\Delta) \dashrightarrow (X',\Delta')$ is a $(K_{X}+\Delta)$-MMP,}\\\text{$R'$ is a $(K_{X'}+\Delta')$-negative extremal ray}\end{array}
\right\}
\end{equation*} 
has at most $m$ elements. 
Then there exists a sequence of at most $(m+1)\rho(X)$ steps of a $(K_{X}+\Delta)$-MMP that terminates with a log minimal model or a Mori fiber space. 
\end{prop}

\begin{proof}
By the argument in \cite[Proof of Proposition 6.2]{hashizumehu}, there exists a sequence of steps of a $(K_{X}+\Delta)$-MMP with scaling of an ample $\mathbb{R}$-divisor $A$
$$(X,\Delta) =: (X_{0},\Delta_{0}) \dashrightarrow (X_{1}, \Delta_{1}) \dashrightarrow \cdots \dashrightarrow (X_{i},\Delta_{i}) \dashrightarrow \cdots$$
such that if we put
$$\lambda_{i}:={\rm inf}\{ \mu \in \mathbb{R}_{\geq 0}\,|\,\text{$K_{X_{i}}+\Delta_{i}+\mu A_{i}$ is nef}\}$$
for each $i \geq 0$, then we have $\lambda_{i-1}>\lambda_{i}$ for every $0<i\leq k$ as long as the sequence $(X,\Delta) \dashrightarrow (X_{k},\Delta_{k})$ contains only flips. 
Here, we used the $\mathbb{Q}$-factoriality of $X$ to apply \cite[Proof of Proposition 6.2]{hashizumehu}. 
If the MMP terminates after at most $m+1$ steps, there is nothing to prove. 
Otherwise, we show that $(X,\Delta) \dashrightarrow (X_{m+1},\Delta_{m+1})$ contains at least one divisorial contraction, where $m$ is as in Proposition \ref{prop--num-step-mmp}. 
For each $0 \leq i \leq m$, let $R_{i}$ be the $(K_{X_{i}}+\Delta_{i})$-negative extremal ray that defines $(X_{i},\Delta_{i})\dashrightarrow (X_{i+1},\Delta_{i+1})$. 
Then $(K_{X_{i}}+\Delta_{i}+\lambda_{i}A_{i})\cdot R_{i}=0$. 
This shows 
$$\lambda_{i}=-\frac{\zeta_{X_{i},R_{i}}(K_{X}+\Delta)}{\zeta_{X_{i},R_{i}}(A)}.$$
Since $\mathcal{S}$ has at most $m$ elements, there exist indices $j$ and $j'$ such that $\zeta_{X_{j},R_{j}}=\zeta_{X_{j'},R_{j'}}$, and therefore $\lambda_{j}=\lambda_{j'}$. 
If $(X,\Delta) \dashrightarrow (X_{m+1},\Delta_{m+1})$ contains only flips, then the fact $\lambda_{j}=\lambda_{j'}$ contradicts $\lambda_{0}>\lambda_{1}>\cdots >\lambda_{m}$. 
Thus, $(X,\Delta) \dashrightarrow (X_{m+1},\Delta_{m+1})$ contains at least one divisorial contraction. 

Now $\rho(X)>\rho(X_{m+1})$ or the MMP terminates after at most $m+1$ steps. 
We may assume $\rho(X)>\rho(X_{m+1})$.  
We repeat the same argument as in the first paragraph with $(X_{m+1},\Delta_{m+1})$. 
 By the argument in \cite[Proof of Proposition 6.2]{hashizumehu}, there exists a sequence of steps of a $(K_{X_{m+1}}+\Delta_{m+1})$-MMP with scaling of an ample $\mathbb{R}$-divisor $H_{m+1}$
$$(X_{m+1},\Delta_{m+1}) \dashrightarrow (X_{m+2}, \Delta_{m+2}) \dashrightarrow \cdots \dashrightarrow (X_{i},\Delta_{i}) \dashrightarrow \cdots$$
such that if we put
$$\nu_{i}:={\rm inf}\{ \mu \in \mathbb{R}_{\geq 0}\,|\,\text{$K_{X_{i}}+\Delta_{i}+\mu H_{i}$ is nef}\}$$
for each $i \geq m+1$, then $\nu_{i-1}>\nu_{i}$ for every $m+1 < i \leq k$ as long as the sequence  $(X_{m+1},\Delta_{m+1}) \dashrightarrow (X_{k},\Delta_{k})$ contains only flips. 
By the same argument as above, we see that $(X_{m+1},\Delta_{m+1}) \dashrightarrow (X_{2m+2},\Delta_{2m+2})$ contains at least one divisorial contraction. 

By repeating this discussion, we obtain the desired conclusion. 
\end{proof}

\section{Main results and applications}\label{sec--prf}

In this section we prove the main results and study their applications.

\begin{thm}\label{thm--ind-finiteness}
With notation as in Definition \ref{defn--ind-small-family}, for any given projective $\mathbb{Q}$-factorial variety $Y_{0}$, the set $\mathfrak{F}_{Y_{0}}$ is finite up to isomorphisms compatible with the markings.
\end{thm}

\begin{proof}
We prove the theorem by induction on $d:={\rm dim}\, Y_{0}$. 
We may assume $d>0$ because otherwise the theorem is obvious. 
We assume Theorem \ref{thm--ind-finiteness} for given source varieties whose dimension is strictly less than $d$. 
We divide the proof into several steps.

\begin{step1}\label{step1}
In this step we collect some notation used in the proof. 

Choose an effective big $\mathbb{Q}$-divisor
$L_{Y_{0}}=\sum_{i=1}^{n}a_{i}L_{i}$
on $Y_{0}$ such that each $L_{i}$ is an effective big $\mathbb{Q}$-divisor, $a_{i}\in\mathbb{Q}_{>0}$, and the rational polytope spanned by $L_{1},\,\cdots,\,L_{n}$ in $N^{1}(Y_{0})_{\mathbb{R}}$ has dimension $\rho(Y_{0})$. 
Let $\alpha\in\mathbb{Q}_{>0}$ be as in \textup{(\ref{defn--ind-small-family-(*)})} for $L_{Y_{0}}$. 
For every
$$\psi\colon Y_{0}\dashrightarrow Y\in\mathfrak{F}_{Y_{0}},$$
put
$$L_{Y}:=\psi_{*}L_{Y_{0}} \qquad {\rm and} \qquad \Delta_{Y}:=\alpha L_{Y}.$$
Then $(Y,2\Delta_{Y})$ is klt, and in particular $(Y,\Delta_{Y})$ is a projective $\mathbb{Q}$-factorial klt pair with big boundary.
By Theorem \ref{thm--base-finite-intnum-func}, the set
\begin{equation*}
\mathcal{S}:=\left\{\zeta_{Y',R'}\circ \psi_{*}\,\,\middle|
\begin{array}{rl}
\bullet \!\!&\text{$\psi \colon Y_{0} \dashrightarrow Y \in \mathfrak{F}_{Y_{0}}$,}\\
\bullet \!\!&\text{$(Y,\Delta_{Y}) \dashrightarrow (Y', \Delta_{Y'})$ is a $(K_{Y}+ \Delta_{Y})$-MMP,}\\
\bullet \!\!&\text{$R'$ is a $(K_{Y'}+\Delta_{Y'})$-negative extremal ray}\end{array}
\right\}
\end{equation*}
is finite. 
By Theorem \ref{thm--base-finite-cont-div}, there exists a finite set $\mathcal{Q}_{0}$ of prime divisors on $Y_{0}$ containing every prime divisor contracted by
$$Y_{0}\dashrightarrow Y\dashrightarrow Y'$$
for every such MMP. 
Finally, we denote the number of elements of $\mathcal{S}$ by $m$.  
\end{step1}

\begin{step1}\label{step2}
In this step and the next step, we prove that the possible resulting models are finite. 

Let $\psi\colon Y_{0}\dashrightarrow Y\in\mathfrak{F}_{Y_{0}}$ be an arbitrary element. 
We first consider MMP terminating with log minimal models
$$(Y,\Delta_{Y})\dashrightarrow (\tilde{Y},\tilde{\Delta}).$$
Since $\psi$ is small and $\Delta_{Y}=\alpha\psi_{*}L_{Y_{0}}$, Lemma \ref{lem--finite-log-min-model-1} shows that there are only finitely many birational contractions from $Y_{0}$ defining such log minimal models, up to isomorphisms compatible with the markings. 
Therefore, the resulting log minimal  models, marked by the compositions from $Y_{0}$, form a finite set.
\end{step1}

\begin{step1}\label{step3}
We next consider the case of Mori fiber spaces
$$(Y,\Delta_{Y})
\dashrightarrow
(Y',\Delta_{Y'})
\overset{h}{\longrightarrow}T$$
and the $(K_{Y'}+\Delta_{Y'})$-negative extremal ray $R' \subset \overline{\rm NE}(Y')$ that defines $h$.  
We associate
$$
\mathcal{P}(Y')
:=
\left\{
P\in\mathcal{Q}_{0}
\,\middle|\,
P\text{ is contracted by }Y_{0}\, \overset{\psi}{\dashrightarrow}\, Y \dashrightarrow Y'
\right\}
$$
and the function
$$
\zeta_{Y', R'} \circ\psi_{*}\in\mathcal{S}.
$$
Since both $\mathcal{Q}_{0}$ and $\mathcal{S}$ are finite sets, we see that there are only finitely many possibilities of 
$
\left(
\mathcal{P}(Y'),
\zeta_{Y', R'}\circ\psi_{*}
\right)
$. 

Fix one nonempty class of marked Mori fiber spaces with the same two data. 
Choose one member
$$(Y^{(0)},\Delta_{Y^{(0)}}) \dashrightarrow (Y'_{0},\Delta'_{0}) \overset{h_{0}}{\longrightarrow}T_{0}.$$
Since the condition \textup{(\ref{defn--ind-small-family-(*)})} in Definition \ref{defn--ind-small-family} is preserved after replacing $Y_{0}$ by $Y^{(0)}$, we replace $Y_{0}$ by $Y^{(0)}$ so that $(Y_{0},\Delta_{Y_{0}})$ has a marked Mori fiber space. 
Let
$$(Y,\Delta_{Y}) \dashrightarrow (Y',\Delta_{Y'}) \overset{h'}{\longrightarrow}T'$$
be another member of the fixed class. 
Since $Y'_{0}$ and $Y'$ contract the same prime divisors, the induced birational map $Y'_{0}\dashrightarrow Y'$ is small. 
Moreover, the extremal-ray functions are positively proportional. 
Then Theorem \ref{thm--base-mfs-small} implies that there exists a small birational map $\theta\colon T_{0}\dashrightarrow T'$ compatible with the Mori fiber spaces.
Let $\mathfrak{G}_{T_{0}}$ be the set of all such small birational maps $\theta\colon T_{0}\dashrightarrow T'$ in the fixed class. 
Then ${\rm id}_{T_{0}}\in\mathfrak{G}_{T_{0}}$ since we have $(Y_{0},\Delta_{Y_{0}}) \dashrightarrow (Y'_{0},\Delta'_{0}) \to T_{0}$. 
By Corollary \ref{cor--ind-form-small-family}, we see that
\begin{itemize}
\item
every element of $\mathfrak{G}_{T_{0}}$ is a small birational map to a projective $\mathbb{Q}$-factorial variety, and 
\item
for every effective $\mathbb{R}$-divisor $H_{T_{0}}$ on $T_{0}$, there exists $\beta\in\mathbb{Q}_{>0}$ such that for every $G_{T_{0}}\in|H_{T_{0}}|_{\mathbb{R}}$, the pair $(T',2\beta\theta_{*}G_{T_{0}})$ is klt for every $\theta\colon T_{0}\dashrightarrow T'\in\mathfrak{G}_{T_{0}}$. 
\end{itemize}
Thus, $\mathfrak{G}_{T_{0}}$ satisfies the assumptions of Definition \ref{defn--ind-small-family}. 
Since we have ${\rm dim}\,T_{0}< {\rm \dim} \,Y_{0}$, the induction hypothesis implies that $\mathfrak{G}_{T_{0}}$ is finite. 

Now we show the marked Mori fiber spaces in the fixed class are finite. 
Fix an element $\theta\colon T_{0}\dashrightarrow T'\in\mathfrak{G}_{T_{0}}$. 
If
$$(Y_{1},\Delta_{Y_{1}}) \dashrightarrow (Y'_{1},\Delta_{Y'_{1}})\longrightarrow T' \qquad {\rm and} \qquad (Y_{2},\Delta_{Y_{2}}) \dashrightarrow(Y'_{2},\Delta_{Y'_{2}})\longrightarrow T'$$
are two marked Mori fiber spaces in the fixed class that induce $\theta$, then $Y'_{1}\dashrightarrow Y'_{2}$ is small because $Y_{0} \dashrightarrow Y'_{1}$ and $Y_{0} \dashrightarrow Y'_{2}$ contract the same prime divisors. 
Since $Y'_{1}\dashrightarrow Y'_{2}$ is over $T'$ and both $-(K_{Y'_{1}}+\Delta_{Y'_{1}})$ and $-(K_{Y'_{2}}+\Delta_{Y'_{2}})$ are ample over $T'$,  
Lemma \ref{lem--inv-mfs} implies that $Y'_{1}\dashrightarrow Y'_{2}$ is an isomorphism over $T'$. 
Hence a fixed marked base determines at most one Mori fiber space in the fixed class. 
Since there are only finitely many marked bases and only finitely many classes, the set of all marked Mori fiber spaces is finite. 
\end{step1}

Combining this with the finiteness of the log minimal models, we obtain a finite set $\mathcal{E}$ of possible marked resulting models. 

\begin{step1}\label{step4}
In this final step, we prove the finiteness of $\mathfrak{F}_{Y_{0}}$. 

Let $\mathcal{D}$ be the set of marked models obtained by finitely many steps of MMP from the target varieties of elements of $\mathfrak{F}_{Y_{0}}$.
By Proposition \ref{prop--num-step-mmp}, the target variety of every element of $\mathfrak{F}_{Y_{0}}$ admits an MMP containing at most $(m+1)\rho(Y_{0})$ steps terminating at one of the finitely many resulting models obtained in Step \ref{step2} and Step \ref{step3}, where $m$ is the number of elements of $\mathcal{S}$ (see Step \ref{step1}). 
On the other hand, Theorem \ref{thm--base-finite-intnum-func} shows that only finitely many extremal-ray functions occur on $\mathcal{D}$. 
By Lemma \ref{lem--inv-divcont} and Lemma \ref{lem--inv-flip}, every marked model in $\mathcal{D}$ has only finitely many marked models in $\mathcal{D}$ from which it is obtained by one step of an MMP. 
Starting from $\mathcal{E}$, we add all marked models in $\mathcal{D}$ that reach one of the resulting models in $\mathcal{E}$ by one step of an MMP. 
We then repeat the same procedure for all the marked models already obtained. 
After carrying out this procedure at most $(m+1)\rho(Y_{0})$ times, we obtain a finite set containing every member of $\mathfrak{F}_{Y_{0}}$.
Therefore, $\mathfrak{F}_{Y_{0}}$ is finite up to isomorphisms compatible with the markings.
This completes the induction.
\end{step1}

We finish the proof. 
\end{proof}

\begin{thm}\label{thm--finite-totalMMP}
Let $(X,\Delta)$ be a projective klt pair such that $K_{X}+\Delta$ or $\Delta$ is big. 
Then the set
\begin{equation*}
\left\{\phi \colon (X,\Delta) \dashrightarrow (X',\Delta')\,\,\middle|
\begin{array}{l}\text{$\phi$ is a finite sequence of steps of a $(K_{X}+\Delta)$-MMP\!}\end{array}
\right\}
\end{equation*}
is a finite set up to isomorphisms compatible with the markings.
\end{thm}

\begin{proof}
We first prove that there is no infinite sequence of steps of a $(K_{X}+\Delta)$-MMP. 
Let 
$$(X,\Delta) \dashrightarrow (X_{1},\Delta_{1}) \dashrightarrow \cdots \dashrightarrow (X_{i},\Delta_{i}) \dashrightarrow \cdots$$
be a sequence of steps of a $(K_{X}+\Delta)$-MMP. 
To prove the termination, by taking a lift of the MMP via a small $\mathbb{Q}$-factorialization, we may assume that $(X,\Delta)$ is $\mathbb{Q}$-factorial. 
By using the second assertion of Lemma \ref{lem--perturb-coeff}, we may assume that $\Delta$ is big. 
If this is an infinite sequence, then there exists $i_{0}$ such that the sequence
$$ (X_{i_{0}},\Delta_{i_{0}}) \dashrightarrow \cdots \dashrightarrow (X_{i},\Delta_{i}) \dashrightarrow \cdots$$
 only contains flips. 
We consider the set 
$$\mathfrak{F}:=\{X_{i_{0}} \dashrightarrow X_{i}\, |\, i \geq i_{0}\}.$$
For any effective $\mathbb{R}$-divisor $L_{i_{0}}$ on $X_{i_{0}}$, take a common resolution $f \colon W \to X$ and $f_{i_{0}} \colon W \to X_{i_{0}}$ of $X \dashrightarrow X_{i_{0}}$, and put $L:=f_{*}f^{*}_{i_{0}}L_{i_{0}}$.  
By \cite[Theorem 1.8]{birkar-bab}, there exists $\alpha \in \mathbb{Q}_{>0}$ such that for any $E \in |L|_{\mathbb{R}}$, the pair $(X,\Delta+2\alpha E)$ is klt. 
Moreover, by Corollary \ref{cor--mmp-boundary-polytope} and replacing $\alpha$ if necessary, we may assume that any $(K_{X}+\Delta)$-MMP is also a $(K_{X}+\Delta+2\alpha E)$-MMP for every $E \in |L|_{\mathbb{R}}$. 
Then $(X_{i},\Delta_{i}+2\alpha E_{i})$ is klt for any $E_{i} \in |L_{i}|_{\mathbb{R}}$, where $L_{i}$ is the birational transform of $L_{i_{0}}$ on $X_{i}$. 
In particular, $(X_i,2\alpha E_i)$ is klt, and therefore $\mathfrak{F}$ satisfies the assumptions of Definition \ref{defn--ind-small-family}. 
Then we can apply Theorem \ref{thm--ind-finiteness} to $\mathfrak{F}$, and we get a contradiction because the map $X_{i'} \dashrightarrow X_{i''}$ is not an isomorphism for any $i' \neq i''$. 
Thus, we see that the $(K_{X}+\Delta)$-MMP terminates. 

For every element $(X,\Delta) \dashrightarrow (X',\Delta')$ of the set in Theorem \ref{thm--finite-totalMMP}, the divisor $K_{X'}+\Delta'$ or $\Delta'$ is big. 
Hence, Lemma \ref{lem--perturb-coeff} and \cite[Corollary 3.8.2]{bchm} show that there are only finitely many $(K_{X'}+\Delta')$-negative extremal rays. 
Since each extremal ray determines the next step of the MMP, there are only finitely many possible MMP steps starting from $(X',\Delta')$.

Suppose that the set in Theorem \ref{thm--finite-totalMMP} is infinite. Since there are only finitely many possible first steps, one of them is the first step of infinitely many finite sequences of MMP. Applying the same argument successively, we can construct an infinite sequence of steps of a $(K_X+\Delta)$-MMP. This contradicts the conclusion of the first paragraph of this proof. 
Thus, the set in Theorem \ref{thm--finite-totalMMP} is finite up to isomorphisms compatible with the markings. 
\end{proof}

\begin{proof}[Proof of Theorem \ref{thm--full-termination}]
Since every step of a $(K_{X}+\Delta)$-MMP yields a new marked model of $(X,\Delta)$, by Theorem \ref{thm--finite-totalMMP}, the number of steps of every $(K_X+\Delta)$-MMP is uniformly bounded. 
Hence, Theorem \ref{thm--full-termination} follows. 
\end{proof}

\begin{proof}[Proof of Theorem \ref{thm--fin-mfs}]
This theorem immediately follows from Theorem \ref{thm--finite-totalMMP}. 
\end{proof}

\begin{proof}[Proof of Corollary \ref{cor--intro-termi-sm-gentype}]
It immediately follows from Theorem \ref{thm--full-termination}. 
\end{proof}

\begin{proof}[Proof of Corollary \ref{cor--intro-termi-big-antilog}]
Let $X$ be a normal projective $\mathbb{Q}$-Gorenstein variety such that $(X,0)$ is klt and $-K_{X}$ is big. 
Fix an effective $\mathbb{R}$-divisor $E$ on $X$ such that $E \sim_{\mathbb{R}}-K_{X}$. 
Then there exists $\epsilon \in (0,1)$ such that $(X,\epsilon E)$ is klt. 
Then $\epsilon E$ is big, and any $K_{X}$-MMP is also a $(K_{X} + \epsilon E)$-MMP since $K_{X} + \epsilon E \sim_{\mathbb{R}}(1-\epsilon)K_{X}$.
Therefore, Corollary \ref{cor--intro-termi-big-antilog} follows from Theorem \ref{thm--full-termination}. 
\end{proof}

\begin{proof}[Proof of Corollary \ref{cor--intro-mmp-CY}]
Fix an effective $\mathbb{R}$-divisor $E$ on $X$ such that $E \sim_{\mathbb{R}} D$. 
Then there exists $r \in \mathbb{R}_{>0}$ such that $(X,\Delta+rE)$ is klt. 
Then $\Delta+rE$ is big and $K_{X}+\Delta+rE$ is numerically equivalent to $rD$. 
Hence, any sequence of steps of MMP on $D$ is also a sequence of steps of a $(K_X+\Delta+rE)$-MMP. 
By Theorem \ref{thm--full-termination}, Corollary \ref{cor--intro-mmp-CY} holds.
\end{proof}

\begin{proof}[Proof of Corollary \ref{cor--intro-strong-sarkisov-program}]
For any projective $\mathbb{Q}$-factorial terminal variety $X$, any mobile linear system $\mathcal{H}$ of big divisors that is not necessarily complete, and any $r \in \mathbb{Q}_{>0}$ such that $(X,r\mathcal{H})$ is canonical, \cite[Proposition 2.8]{he-strong-sarkisov} implies the existence of $\Delta \in r\mathcal{H}_{\mathbb{R}}$ such that $(X,\Delta)$ is klt. 
Then $\Delta$ is big, and any sequence of $(K_X+r\mathcal{H})$-flips in the sense of \cite{he-strong-sarkisov} is also a sequence of steps of a $(K_X+\Delta)$-MMP. 
Hence, \cite[Conjecture 1.5]{he-strong-sarkisov} follows from Theorem \ref{thm--full-termination}.
Then Corollary \ref{cor--intro-strong-sarkisov-program} follows from \cite[Theorem 1.8]{he-strong-sarkisov}.
See also \cite{bruno-matsuki} for the underlying arguments. 
\end{proof}

\begin{proof}[Proof of Corollary \ref{cor--intro-MMP-algorithm}]
Let $X$ be a normal projective $\mathbb{Q}$-Gorenstein variety such that $(X,0)$ is klt and $K_{X}$ or $-K_{X}$ is big. 
As in \cite{yasuda-mmp-algorithm}, the only obstructions to the MMP algorithm are the full termination of sequences of non-trivial steps
of MMP in the sense of \cite[Definition 3.5]{has-mmp-normal-pair} (see \cite[Proof of Proposition 6.11]{yasuda-mmp-algorithm}) and the abundance theorem (\cite[Proof of Proposition 3.8]{yasuda-mmp-algorithm}) over the field of complex numbers. 
In our situation, the abundance theorem immediately follows from \cite[Theorem 3.3]{kollar-mori} (see also \cite[Theorem 3.9.1]{bchm}) when $K_{X}$ is big, and we do not need the abundance theorem when $-K_{X}$ is big.
Therefore, we only need to treat the full termination. 

Let 
$$
\xymatrix{
X=X_{1}\ar@{-->}[r]&\cdots \ar@{-->}[r]& X_{i}\ar@{-->}[rr]\ar[dr]_{\varphi_{i}}&& X_{i+1}\ar[dl]^{\varphi'_{i}} \ar@{-->}[r]&\cdots \\
&&&V_{i}
}
$$
be a sequence of steps of a $K_{X}$-MMP in the sense of \cite[Definition 3.5]{has-mmp-normal-pair}, that is, for each $i$, 
\begin{itemize}
\item
$X_{i} \to V_{i}$ is a birational morphism and $X_{i+1} \to V_{i}$ is a small birational morphism, 
\item
$-K_{X_{i}}$ is ample over $V_{i}$, and 
\item
$K_{X_{i+1}}$ is $\mathbb{R}$-Cartier and ample over $V_{i}$.
\end{itemize}
Note that each $X_{i}$ is not necessarily $\mathbb{Q}$-factorial. 
We assume that every $X_{i} \dashrightarrow X_{i+1}$ is not an isomorphism.  
By taking a small $\mathbb{Q}$-factorialization $\tilde{X} \to X$ of $X$ and using the lift of MMP as in \cite[Theorem 3.9]{has-mmp-normal-pair}, we get a sequence of steps of a $K_{\tilde{X}}$-MMP 
$$\tilde{X}=:\tilde{X}_{1} \dashrightarrow \cdots \dashrightarrow \tilde{X}_{i} \dashrightarrow \tilde{X}_{i+1}\dashrightarrow \cdots$$
such that for each $i$, 
\begin{itemize}
\item
there exists a projective birational morphism $f_{i} \colon \tilde{X}_{i} \to X_{i}$ compatible with the $K_{X}$-MMP and $K_{\tilde{X}}$-MMP such that $K_{\tilde{X}_{i}}=f^{*}_{i}K_{X_{i}}$, and 
\item
$\tilde{X}_{i} \dashrightarrow \tilde{X}_{i+1}$ is a finite sequence of steps of a $K_{\tilde{X}}$-MMP. 
\end{itemize}
By construction, $K_{\tilde{X}}$ or $-K_{\tilde{X}}$ is big. 
Moreover, the $K_{X}$-MMP terminates if and only if the $K_{\tilde{X}}$-MMP terminates. 
Hence, Theorem \ref{thm--full-termination} and
Corollary \ref{cor--intro-termi-big-antilog} imply the desired full termination. 
Thus, Corollary \ref{cor--intro-MMP-algorithm} holds. 
\end{proof}

\section{Finiteness of Mori fiber spaces versus full termination}\label{sec--exam}

In this section, we discuss a relation between the finiteness of marked Mori fiber spaces and full termination of MMP for klt pairs. 

\begin{idea}
Here, we describe the idea behind this section. 

Let $(X,\Delta)$ be a projective $\mathbb{Q}$-factorial klt pair and $Y :=X \times \mathbb{P}^{1}$ with the projections $\pi \colon Y \to X$ and $p \colon Y \to \mathbb{P}^{1}$. 
Fix a point $u \in \mathbb{P}^{1}$, and put $S:=p^{*}u$ and $\Gamma:=\frac{1}{2}S+\pi^{*}\Delta$. 
Then it follows that $(Y,\Gamma)$ is a projective $\mathbb{Q}$-factorial klt pair and $\pi \colon (Y,\Gamma) \to X$ is a Mori fiber space. 

Let $(X,\Delta) \dashrightarrow (X',\Delta')$ be a finite sequence of steps of a $(K_{X}+\Delta)$-MMP. 
Then we can construct a diagram
$$
\xymatrix{
(Y,\Gamma)\ar@{-->}[rr]\ar[d]_{\pi}&&(Y',\Gamma')\ar[d]^{\pi'}\\
(X,\Delta)\ar@{-->}[rr]&&(X',\Delta')
}
$$
such that $(Y, \Gamma) \dashrightarrow (Y',\Gamma')$ is a finite sequence of steps of a $(K_{Y}+\Gamma)$-MMP and both $\pi \colon (Y,\Gamma) \to X$ and $\pi' \colon (Y',\Gamma') \to X'$ are Mori fiber spaces. 
From this construction, we see that the finiteness of marked Mori fiber spaces obtained from the $(K_Y+\Gamma)$-MMP implies the finiteness of marked models obtained from the $(K_X+\Delta)$-MMP, and hence full termination. 
\end{idea}

\begin{lem}\label{lem--MMP-product-P^1}
Let $(X,\Delta)$ be a projective $\mathbb{Q}$-factorial klt pair, and let 
$$
\xymatrix{
(X,\Delta)\ar@{-->}[rr]
\ar[dr]_{f}&&(X',\Delta')\ar[dl]^{f'}\\
&V
}
$$
be a step of a $(K_{X}+\Delta)$-MMP. 
We put $Y:=X \times \mathbb{P}^{1}$, $Y':=X' \times \mathbb{P}^{1}$, and $W:=V \times \mathbb{P}^{1}$. 
Fix a point $u \in \mathbb{P}^{1}$, and let $S$, $S'$, and $T$ be the sections of $Y \to X$, $Y' \to X'$, and $W \to V$ defined by $u$, respectively. 
In other words, we set $S:=X \times \{u\}$, $S':=X' \times \{u\}$, and $T:= V \times \{u\}$. 
We define an $\mathbb{R}$-divisor $\Gamma$ on $Y$ by $\Gamma:= \frac{1}{2}S+\Delta \times \mathbb{P}^{1}$, where $\Delta \times \mathbb{P}^{1}$ is the pullback of $\Delta$ to $Y$. 
Similarly, we define an $\mathbb{R}$-divisor $\Gamma'$ on $Y'$ by 
$\Gamma':= \frac{1}{2}S'+\Delta' \times \mathbb{P}^{1}$. 
Then the induced diagram
 $$
\xymatrix{
(Y,\Gamma)\ar@{-->}[rr]
\ar[dr]_{g}&&(Y',\Gamma')\ar[dl]^{g'}\\
&W
}
$$
is a step of a $(K_{Y}+\Gamma)$-MMP and the morphism $(Y',\Gamma') \to X'$ is a Mori fiber space. 
\end{lem}

\begin{proof}
Let $p\colon W \to \mathbb{P}^{1}$ be the projection. 
Note that $p \circ g\colon Y \to \mathbb{P}^{1}$ and $p \circ g' \colon Y' \to \mathbb{P}^{1}$ are projections. 
Let $q_{Y} \colon Y \to X$ and $q_{Y'} \colon Y' \to X'$ be projections. 

We first show that $Y$ is $\mathbb{Q}$-factorial. 
Fix an arbitrary Weil divisor $D$ on $Y$. 
We apply \cite[III, Exercise 12.5]{hartshorne} after restricting $D$ to the smooth locus of $Y$. 
Then we can find $\alpha \in \mathbb{Z}$ and a Weil divisor $D_{X}$ on $X$ such that $D \sim \alpha S +q_{Y}^{-1}(D_{X})$. 
Since $X$ is $\mathbb{Q}$-factorial, we have $q_{Y}^{-1}(D_{X})=q_{Y}^{*}D_{X}$, which is $\mathbb{Q}$-Cartier. 
Hence, $D$ is also $\mathbb{Q}$-Cartier, and therefore $Y$ is $\mathbb{Q}$-factorial.

It is clear that $g \colon Y \to W$ is a birational morphism and $g' \colon Y' \to W$ is a small birational morphism. 
Since $S=g^{*}T$, we have
\begin{equation*}
\begin{split}
K_{Y}+\Gamma&=q_{Y}^{*}K_{X}+(p \circ g)^{*}K_{\mathbb{P}^{1}}+\frac{1}{2}S+q_{Y}^{*}\Delta \\
&=q_{Y}^{*}(K_{X}+\Delta)+g^{*}\left(p ^{*}K_{\mathbb{P}^{1}}+\frac{1}{2}T\right)\sim_{\mathbb{R},\,W}q_{Y}^{*}(K_{X}+\Delta). 
\end{split}
\end{equation*}
Because $-(K_{X}+\Delta)$ is ample over $V$ and we can regard $g \colon Y \to W$ as the base change of $X \to V$ by $W \to V$, we see that $-(K_{Y}+\Gamma)$ is ample over $W$. 
Similarly, we have 
\begin{equation*}
\begin{split}
K_{Y'}+\Gamma'&=q_{Y'}^{*}K_{X'}+(p \circ g')^{*}K_{\mathbb{P}^{1}}+\frac{1}{2}S'+q_{Y'}^{*}\Delta' \\
&=q_{Y'}^{*}(K_{X'}+\Delta')+g'^{*}\left(p ^{*}K_{\mathbb{P}^{1}}+\frac{1}{2}T\right) \sim_{\mathbb{R},\,W}q_{Y'}^{*}(K_{X'}+\Delta'). 
\end{split}
\end{equation*}
Because $K_{X'}+\Delta'$ is ample over $V$ and we can regard $g' \colon Y' \to W$ as the base change of $X' \to V$ by $W \to V$, we see that $K_{Y'}+\Gamma'$ is ample over $W$. 

It follows that $\rho(Y/V)=2$ since $\rho(X/V)=1$ and $Y = X \times \mathbb{P}^{1}$.   
Then 
$$1+1 \leq \rho(W /V)+\rho(Y/W) \leq \rho(Y/V)=2,$$
and therefore $\rho(Y/W)=1$. 
From these arguments, we see that the diagram
$$
\xymatrix{
(Y,\Gamma)\ar@{-->}[rr]
\ar[dr]_{g}&&(Y',\Gamma')\ar[dl]^{g'}\\
&W
}
$$
is a step of a $(K_{Y}+\Gamma)$-MMP. 
Finally, we have $\rho(Y'/X')=1$ since $Y'=X' \times \mathbb{P}^{1}$, and 
$$K_{Y'}+\Gamma'=q_{Y'}^{*}(K_{X'}+\Delta')+g'^{*}\left(p ^{*}K_{\mathbb{P}^{1}}+\frac{1}{2}T\right)\sim_{\mathbb{R},\,X'}-\frac{3}{2}S'.$$
Therefore, $-(K_{Y'}+\Gamma')$ is ample over $X'$, from which we see that $q_{Y'} \colon (Y',\Gamma') \to X'$ is a Mori fiber space. 
Note that $Y'$ is $\mathbb{Q}$-factorial since $Y$ is $\mathbb{Q}$-factorial.  
\end{proof}

\begin{thm}\label{thm--mfsnumber-mmpnumber}
Let $(X,\Delta)$ be a projective $\mathbb{Q}$-factorial klt pair and $Y :=X \times \mathbb{P}^{1}$. 
We put $S:=X \times \{u\}$ for a point $u \in \mathbb{P}^{1}$, and we put $\Gamma:=\frac{1}{2}S+\Delta \times \mathbb{P}^{1}$. 
Suppose that the set of marked Mori fiber spaces arising from the $(K_Y+\Gamma)$-MMP is finite up to isomorphisms compatible with the markings. 
Then the set of marked models arising from the $(K_X+\Delta)$-MMP is also finite up to isomorphisms compatible with the markings. 
In particular, every $(K_X+\Delta)$-MMP terminates.
\end{thm}

\begin{proof}
Suppose for contradiction that there are infinitely many marked models 
$$\phi_i\colon (X,\Delta)\dashrightarrow (X_{i},\Delta_{i}),$$
arising from the $(K_X+\Delta)$-MMP such that no two of them are isomorphic compatibly with the markings.
By Lemma \ref{lem--MMP-product-P^1}, each $\phi_{i}$ induces a marked Mori fiber space
$$(Y,\Gamma)\dashrightarrow (Y_{i},\Gamma_{i})\to X_{i}.$$
By the assumption, there exist $i \neq j$ and an isomorphism $\tau \colon Y_{i} \to Y_{j}$ that is compatible with the markings. 
Let $S_{i}$ and $S_{j}$ be the birational transforms of $S$ on $Y_{i}$ and $Y_{j}$, respectively. 
Then $\tau_{*}S_{i}=S_{j}$. 
Since $S_i\cong X_i$ and $S_j\cong X_j$, the restriction $\tau|_{S_i}$ induces an isomorphism $\psi\colon X_i\to X_j$, and compatibility with the markings gives $\psi\circ\phi_i=\phi_j$. 
This is a contradiction. 
Hence, the set of marked models is finite up to isomorphisms compatible with the markings. 
The final assertion is clear from this discussion. 
\end{proof}

The following example shows that the finiteness of marked Mori fiber spaces does not hold in general.

\begin{exam}\label{exam--infinite-marked-mfs}
Let $X$ be the smooth surface in \cite[Example 4-2-3]{kawamata-matsuda-matsuki}. 
Then $K_{X}$ is not pseudo-effective and $X$ contains infinitely many $(-1)$-curves $C_i$. 
For each $i$, the contraction $X \to X_{i}$ of $C_i$ is a step of the $K_{X}$-MMP. 
By continuing this MMP, we get a Mori fiber space
$$X \overset{\phi_{i}}{\dashrightarrow} X'_{i}\longrightarrow Z_{i}$$
such that $C_i$ is contracted by $\phi_{i}$. 

Suppose that only finitely many marked Mori fiber spaces arise from the $K_X$-MMP up to isomorphisms compatible with the markings.
Then only finitely many prime divisors on $X$ can be contracted by their markings.
This contradicts the fact that $C_i$ is contracted by $\phi_i$ for every $i$.
Thus, infinitely many marked Mori fiber spaces arise from the $K_X$-MMP.
\end{exam}

%%%%%%%%%%%%%%%

\end{document}